\documentclass[12pt]{amsart}
\usepackage{}

\usepackage{amsmath}
\usepackage{amsfonts}
\usepackage{amssymb}
\usepackage[all]{xy}           %xypic macro for latex2.09

\usepackage{bbding}
\usepackage{txfonts}
\usepackage{amscd}

\usepackage[shortlabels]{enumitem}
\usepackage{ifpdf}
\ifpdf
  \usepackage[colorlinks,final,backref=page,hyperindex]{hyperref}
\else
  \usepackage[colorlinks,final,backref=page,hyperindex]{hyperref}
\fi
\usepackage{tikz}
\usepackage[active]{srcltx}

\makeatletter

\newtheorem{thm}{Theorem}[section]
\newtheorem{lem}[thm]{Lemma}
\newtheorem{cor}[thm]{Corollary}
\newtheorem{pro}[thm]{Proposition}
\newtheorem{ex}[thm]{Example}
\newtheorem{rmk}[thm]{Remark}
\newtheorem{defi}[thm]{Definition}

\newcommand {\emptycomment}[1]{}

\newcommand{\dr}{\delta^{\rm reg}}

\newcommand{\be }{\begin{equation}}
\newcommand{\ee }{\end{equation}}

\newcommand{\g}{\mathfrak g}
\newcommand{\h}{\mathfrak h}

\newcommand{\Real}{\mathbb R}

\newcommand{\huaC}{{\mathcal{C}}}%{\mathcal{C}}

\newcommand{\huaO}{{\mathcal{O}}}

\newcommand{\frkd}{\mathfrak d}

\newcommand{\frkg}{\mathfrak g}

\newcommand{\frks}{\mathfrak s}

\newcommand{\frkB}{\mathfrak B}

\newcommand{\frkL}{\mathfrak L}

\newcommand{\frkR}{\mathfrak R}

\newcommand{\half}{\frac{1}{2}}

\newcommand{\Courant}[1]{\left\llbracket  #1\right\rrbracket }

\newcommand{\Id}{\rm{id}}

\newcommand{\br}[1]{   [ \cdot,    \cdot  ]   }
\newcommand{\id}{\rm{id}}

\newcommand{\dM}{\mathrm{d}}

\newcommand{\Hom}{\mathrm{Hom}}

\newcommand{\Sym}{\mathrm{Sym}}

\newcommand{\gl}{\mathfrak {gl}}

\newcommand{\AD}{\mathfrak{ad}}

\newcommand{\End}{\mathrm{End}}
\newcommand{\ad}{\mathrm{ad}}

\newcommand{\K}{\mathbb{K}}

\begin{document}

\title[From pre-Lie bialgebras to phase spaces of Lie algebras]{From pre-Lie bialgebras to phase spaces of Lie algebras: a categorical correspondence}

\author{Qi Wang}
\address{School of Mathematics and Statistics, Changchun University of Technology, Changchun 130012, Jilin, China}
\email{wangqi2018@ccut.edu.cn}

\author{Xueyi Wang}
\address{School of Mathematics and Statistics, Northeast Normal University, Changchun 130024, China}
\email{wangxueyi@nenu.edu.cn}
\vspace{-5mm}

\author{Jiefeng Liu}
\address{School of Mathematics and Statistics, Northeast Normal University, Changchun 130024, China}
\email{liujf534@nenu.edu.cn}
\vspace{-5mm}

%\date{\today}

\begin{abstract}
This paper establishes a categorical framework for phase spaces of Lie algebras, pre-Lie
bialgebras, Manin triples, classical $\mathfrak s$-matrices, and relative Rota-Baxter operators by introducing the concept of coherent homomorphisms. Starting with endo pre-Lie algebras (pre-Lie algebras equipped with endomorphisms), we extend classical constructions to this enhanced setting, which leads to the notion of coherent endomorphisms for each class of structures.  Through polarization, these endomorphisms naturally generalize to coherent homomorphisms, establishing well-defined categories of these algebraic objects. Furthermore, mappings between categories are elevated to functors or equivalences, formalizing interconnections among the constructions. Finally, exploiting the categorical correspondence between s-matrices and relative Rota-Baxter operators, we develop cohomology and deformation of $\mathfrak s$-matrices, thereby bridging algebraic structures with category-theoretic methods.
\end{abstract}

%\subjclass[2010]{17B10, 17B56, 17A42}

\keywords{phase space, pre-Lie bialgebra, Manin triple, relative Rota-Baxter operator, homomorphism, $\frks$-matrix}
\footnotetext{{\it{MSC}}: 16T10, 17A30, 17B30}
\maketitle

%\tableofcontents

%\setcounter{section}{0}

\allowdisplaybreaks

%\end{document}

\section{Introduction}
This paper presents a categorical investigation of phase spaces, pre-Lie bialgebras, Manin triples, classical $\frks$-matrices, and relative Rota-Baxter operators through the systematic introduction of their respective morphisms. By demonstrating that these morphisms preserve natural correspondences among these classes, we elevate the interrelations between these structures to the level of categorical correspondences. This work bridges algebraic structure theory with category-theoretic methods, demonstrating how categorical perspectives can systematically organize and deepen understanding of classical algebraic constructions.
\subsection{Phase spaces of Lie algebras, pre-Lie bialgebras and the related structures}
Pre-Lie algebras (or left-symmetric algebras) are a class of nonassociative algebras coming from the study of convex homogeneous cones, affine manifolds and affine structures on Lie groups, deformation of associative algebras and then  appeared in many fields in mathematics and mathematical physics, such as complex and symplectic structures on Lie groups and Lie algebras
(\cite{BBM,Chu,Kan}),  Poisson brackets and infinite dimensional Lie algebras (\cite{BaN}), vertex algebras (\cite{Bakalov}), $F$-manifold algebras (\cite{Dot19,LSB}), homotopy algebra structures (\cite{Ban,DSV}) and operads (\cite{CL}). See  the survey articles \cite{Bai21} and the references therein for more details.

The concept of phase space associated with a Lie algebra was initially introduced by Kupershmidt in \cite{Ku94} and subsequently generalized in \cite{Bai06}. Kupershmidt demonstrated that pre-Lie algebras constitute the fundamental structures underlying the phase spaces of Lie algebras, forming a natural category that holds significant importance in both classical and quantum mechanics (\cite{Ku99a}). Furthermore, phase spaces can be precisely characterized as para-K\"{a}hler structures on Lie algebras. From a geometric perspective, a para-K\"{a}hler manifold is defined as a symplectic manifold equipped with a pair of transversal Lagrangian foliations. Specifically, a para-K\"{a}hler Lie algebra represents the Lie algebra of a Lie group $G$ endowed with a  $G$-invariant para-K\"{a}hler structure (\cite{Kan}).  In \cite{Left-symmetric bialgebras}, Bai established a fundamental result showing that every para-K\"{a}hler Lie algebra is isomorphic to a phase space of a Lie algebra. For further exploration of phase spaces, para-K\"{a}hler Lie algebras, and their diverse applications, we refer readers to \cite{And,Andrada,BajoBena,BeM,Cal,Cal1}.

Building upon the investigation of phase spaces of Lie algebras through  pre-Lie algebras, the concepts of pre-Lie bialgebras (also referred to as left-symmetric bialgebras), matched pairs, and Manin triples for pre-Lie algebras were systematically developed in \cite{Left-symmetric bialgebras}. Notably, the introduction of coboundary pre-Lie bialgebras naturally leads to a pre-Lie algebraic analogue of the classical Yang-Baxter equation (CYBE). Solutions to this equation, known as classical $\frks$-matrices, provide a direct construction of both pre-Lie bialgebras and phase spaces of Lie algebras. A particularly interesting connection emerges through Hessian structures on pre-Lie algebras, which correspond to affine Lie groups  $G$ equipped with $G$-invariant Hessian metrics (\cite{Shima}). These structures yield non-degenerate symmetric classical $\frks$-matrices, establishing a profound link between geometric and algebraic constructions. In \cite{Ku99b}, Kupershmidt introduced the notion of a relative Rota-Baxter operator (also called an $\huaO$-operator) on a Lie algebra in order to better understand the relationship between the classical Yang-Baxter equation and the related integrable systems. A symmetric classical  $\frks$-matrix can be equivalently described by a relative Rota-Baxter operator with respective to the coregular representation. For a comprehensive treatment of Rota-Baxter operators, we refer to the monograph \cite{Gub}. The intricate relationships among these mathematical structures can be effectively summarized in the following diagram:

\vspace{-.1cm}
\begin{equation}
    \begin{split}
        \xymatrix{
            &&\text{matched pairs of}\atop \text{pre-Lie algebras}& \\
                  & 
             \text{solutions of}\atop \text{CYBE} 
            \ar@2{->}[r]& \text{pre-Lie}\atop \text{ bialgebras}
            \ar@2{<->}[r] \ar@2{<->}[d] \ar@2{<->}[u]& \text{phase spaces}\atop \text{of Lie algebras}.\\
            &&\text{Manin triples for}\atop \text{pre-Lie algebras} & }
    \end{split}
    \label{eq:bigdiag}
\end{equation}
\vspace{-.6cm}

\subsection{Morphisms of bialgebras, Manin triples and matched pairs} 
To advance the study and applications of these important structures and their interrelations as depicted in the diagram, it is crucial to understand them within the framework of category theory. The key step in this direction involves defining appropriate morphisms for these structures to organize them into well-defined categories. Such a categorical formulation would enable the precise characterization of their relationships through functors and equivalences between these categories. However, this endeavor presents significant challenges. While one can define a notion of homomorphisms for pre-Lie bialgebras and Manin triples for pre-Lie algebras by analogy with the corresponding homomorphisms of Lie bialgebras and Manin triples for Lie algebras (as presented in \cite{CP}), this approach encounters fundamental limitations. Specifically, the homomorphisms of Manin triples fail to align consistently with the homomorphisms of their corresponding pre-Lie bialgebras. This discrepancy highlights the need for a more refined approach to defining morphisms that can properly capture the categorical structure of these mathematical objects.

\emptycomment{For further studies and applications of the important structures and their relations in the above diagram, it is significant to understand them in the context of categories. The crucial step is to define suitable morphisms for these structures to make them into categories, so that their relations can be made precise as functors and equivalences among these categories. Unfortunately, the usually defined morphisms for this purpose fail. A notion of usual homomorphisms of pre-Lie bialgebras and Manin triples for pre-Lie algebras can be defined in analog to homomorphisms of Lie bialgebras and Manin triples for Lie algebras in \cite{CP}. However, the homomorphisms of Manin triples do not agree with the homomorphisms of their corresponding Lie pre-bialgebras. }

To address the challenges at the Lie algebra level, the authors in \cite{BGS} utilize two key strategies. The first involves a polarization process that shifts the focus from homomorphisms to endomorphisms. The second strategy entails reordering operations to endow Lie algebras with additional structures, specifically bialgebras and endomorphisms. This polarization strategy can be likened to transforming a homogeneous polynomial in a single variable into a multilinear form in multiple variables through linear substitutions and derivations \cite{Pr}. Rather than attempting to define morphisms between arbitrary Lie bialgebras, the authors concentrate on the specific case of endomorphisms for a given object. To define these endomorphisms for a Lie bialgebra, they treat it as an additional endomorphism structure layered onto the Lie bialgebra, which itself is derived from imposing a bialgebra structure on a Lie algebra. The crux of their approach lies in reversing the sequence of these two processes: they first augment Lie algebras with endomorphisms, resulting in what they term { endo Lie algebras}, and subsequently introduce a bialgebra structure to these endo Lie algebras. This sequence is illustrated in the following diagram, specifically for the case of Lie bialgebras.

\emptycomment{To overcome the problem in the Lie algebra level, the authors in \cite{BGS} utilize two strategies. The first strategy is a polarization process that allows them to first consider endomorphisms instead of homomorphisms. The second strategy is a change of order of the operations which equip Lie algebras with two extra structures: bialgebras and endomorphisms. Their strategy of polarization is in a sense similar to polarizing a homogeneous polynomial in one variable to a multilinear form in multivariable by linear substitutions and derivations \cite{Pr}.  So for each of the classes of constructions such as Lie bialgebras, instead of attempting to define morphisms between any two of them, they focus on the special case of endomorphisms on any given object. Furthermore, to define endomorphisms for a Lie bialgebra, they regard it as equipping an extra endomorphism structure to the  Lie bialgebra which, by itself, is obtained from equipping a bialgebra structure to a Lie algebra. The strategy for this purpose is to switch the order of these two processes. They first equip Lie algebras with endomorphisms, called { endo Lie algebras} and then  introduce a bialgebra structure for endo Lie algebras. This is shown in the following diagram for the instance of Lie bialgebras. }

\vspace{-0.5cm}
\begin{equation}
    \begin{split}
        \xymatrix{
            \text{\small Lie algebras} \ar[rrr]^{\text{\small endomorphism}} \ar[d]^{\text{\small bialgebraization}} &&& \text{\small Endo Lie algebras} \ar[d]^{\text{\small bialgebraization}} \\
            \text{\small Lie bialgebras} \ar[rrr]^{\text{\small endomorphism}}
            &&& {\text{\small Endo Lie bialgebras}}
        .}
    \end{split}
    \label{eq:diag}
\end{equation}
\vspace{-0.3cm}

With this framework, each structure (such as Lie bialgebras, Manin triples, and matched pairs) on endo Lie algebras naturally leads to a notion of endomorphisms for the corresponding structures on Lie algebras. Once this is established, a polarization process, as previously mentioned, extends the concept of endomorphisms to homomorphisms, thereby equipping each class of objects with a categorical structure.

Guided by this philosophy, we first introduce the notion of an endo pre-Lie algebra, which is a pre-Lie algebra endowed with a pre-Lie algebra endomorphism. We then investigate matched pairs, Manin triples, bialgebras for endo pre-Lie algebras, and endo phase spaces for Lie algebras, which in turn induce endomorphisms of these structures on pre-Lie algebras. Subsequently, a polarization process generalizes the notion of endomorphisms to homomorphisms, endowing each class of objects with a category structure. The relationships between these categories are summarized in the following diagram:

\emptycomment{With this ideal, each of the structures (like Lie bialgebras, Manin triples, mathched pairs) on endo Lie algebras naturally gives rise to a notion of endomorphisms of this structure on Lie algebras. Once this is done, then as noted above, a polarization process extends the notion of endomorphisms to a notion of homomorphisms, equipping each class of objects with a category structure.  

Following this philosophy, we first introduce the notion of an endo pre-Lie algebra, which is a pre-Lie algebra equipped with a pre-Lie algebra endomorphism. Then we study the matched pairs, Manin triples and bialgebras for endo pre-Lie algebras, which induce the endomorphisms of this structure on pre-Lie algebras. Next, a polarization process extends the notion of endomorphisms to a notion of homomorphisms, equipping each class of objects with a category structure. The relations between these categories are summarized in the following  diagram. }

\vspace{-.5cm}
\begin{equation}
\label{eq:bigdiagcat}
\begin{split}
 \xymatrix{
 && {\begin{subarray}{c} \text{category of}\\ \text{matched pairs of}\\ \text{pre-Lie algebras} \end{subarray}}&&\\
{\begin{subarray}{c} \text{category of } \\ \text{solutions of}\\ \text{CYBE}
\end{subarray}} 
 \ar@2{->}[rr]&& {\begin{subarray}{c} \text{category of} \\ \text{pre-Lie bialgebras}\end{subarray}} \ar@2{<->}[d] \ar@2{<->}[u] \ar@2{<->}[rr]&&{\begin{subarray}{c} \text{category of } \\ \text{phase spaces}.
\end{subarray}}\\
&& {\begin{subarray}{c} \text{category of} \\ \text{Manin triples for}\\ \text{pre-Lie algebras}\end{subarray}}
&&}
\end{split}
\end{equation}
\vspace{-.5cm}

\subsection{Morphism of phase spaces, classical $\frks$-matrices and relative Rota-Baxter operators}

A phase space of a Lie alebra $\g$ is a natural symplectic structure $\omega_p$ on the Lie algebra $\g\oplus\g^*$ such that both $\g$ and $\g^*$ are Lie subalgebras. A natural approach to define endomorphisms between phase spaces is to use symplectic endomorphisms. However, this approach fails to establish a categorical correspondence between phase spaces of Lie algebras and Manin triples for pre-Lie algebras. In \cite{Ku99b}, Kupershmidt demonstrated that a symplectic structure on a Lie algebra is equivalent to an invertible skew-symmetric relative Rota-Baxter operator with respect to the coadjoint representation. To address this, we introduce a coherent homomorphism between symplectic Lie algebras, which establishes an equivalence between the category of symplectic Lie algebras with coherent homomorphisms and the category of invertible skew-symmetric relative Rota-Baxter operators with their natural homomorphisms. Notably, this coherent homomorphism between phase spaces arises as the extended endomorphism from the endo phase spaces.

Similarly, a symmetric classical $\frks$-matrix on a pre-Lie algebra is equivalent to a symmetric relative Rota-Baxter operator on its sub-adjacent Lie algebra with respect to the coregular representation. From a categorical perspective, we introduce another homomorphism between $\frks$-matrices, ensuring that the induced categories are equivalent. Recently, the deformation theory of relative Rota-Baxter operators on Lie algebras has been developed in \cite{TBGS}. The authors provide a cohomology theory for relative Rota-Baxter operators and use it to study their deformations. In this paper, leveraging the close relationship between $\frks$-matrices and relative Rota-Baxter operators, we develop a cohomology theory for $\frks$-matrices on pre-Lie algebras. Furthermore, using this new homomorphism between symmetric classical $\frks$-matrices, we investigate the deformations of $\frks$-matrices.

\subsection{Outline of the paper}

The paper is organized as following. In Section \ref{sec:pre}, we first recall the representations and cohomology of pre-Lie algebras. Then we recall the graded Lie algebra whose Maurer-Cartan elements are relative Rota-Baxter operators on Lie algebras and the cohomology for a relative Rota-Baxter operator.

In Section \ref{sec:endoprelie}, we introduce the notion of an endo pre-Lie algebra and develop a bialgebra theory for endo pre-Lie algebras together with their equivalences to  matched pairs, Manin triples of endo pre-Lie algebras and endo phase phases (Theorem \ref{thm:equvialent}), which leads to the notion of coherent endomorphisms for each class of structures.

In Section \ref{sec:coherent endoprelie}, we interpret a bialgebra of endo pre-Lie algebras as an endomorphism  of a pre-Lie bialgebra, which then naturally generalizes to a homomorphism of pre-Lie bialgebras that is compatible with that of phase spaces of Lie algebras, Manin triples and   matched pairs of pre-Lie algebras, showing that the correspondences among pre-Lie bialgebras, phase spaces of Lie algebras, Manin triples for pre-Lie algebras  and matched pairs of pre-Lie algebras are equivalences of categories (Theorem \ref{thm:main theorem}). In particular, we build the equivalences of categories between symplectic Lie algebras and corresponding relative Rota-Baxter operators (Proposition \ref{pro:weak-homomorphsim symp}).

In Section \ref{sec:def-phase}, we introduce the weak homomorphism between symmetric classical $\frks$-matrices and show that the equivalences of categories between symmetric classical $\frks$-matrices and corresponding relative Rota-Baxter operators (Proposition \ref{pro:weak-homomorphsim}). We use the graded Lie algebra for relative Rota-Baxter operators on Lie algebras to construct a graded Lie algebra whose Maurer-Cartan elements characterize symmetric classical $\frks$-matrices on pre-Lie algebras (Theorem \ref{thm:MC char}).  Furthermore, we use the graded Lie algebra for $\frks$-matrices to establish the cohomology of $\frks$-matrices and use the this cohomology and weak homomorphisms to study their deformations. 

 In this paper, all the vector spaces are over algebraically closed field $\mathbb K$ of characteristic $0$ and finite dimensional.

\vspace{2mm}
 \noindent {\bf Acknowledgement:} This research is supported by  NSFC (12201068, 12371029, W2412041) and the National Key Research and Development Program of China (2021YFA1002000).

\section{Preliminaries} \label{sec:pre}

\subsection{Pre-Lie algebras, representations and cohomologies}
\begin{defi}  A {\bf pre-Lie algebra} is a pair $(\g,\cdot_\g)$, where $\g$ is a vector space and  $\cdot_\g:\g\otimes \g\longrightarrow \g$ is a bilinear multiplication
satisfying that for all $x,y,z\in \g$, the associator
$(x,y,z)=(x\cdot_\g y)\cdot_\g z-x\cdot_\g(y\cdot_\g z)$ is symmetric in $x,y$,
i.e.
$$(x,y,z)=(y,x,z),\;\;{\rm or}\;\;{\rm
equivalently,}\;\;(x\cdot_\g y)\cdot_\g z-x\cdot_\g(y\cdot_\g z)=(y\cdot_\g x)\cdot_\g
z-y\cdot_\g(x\cdot_\g z).$$
\end{defi}

Let $(\g,\cdot_\g)$ be a pre-Lie algebra. The commutator $
[x,y]_\g=x\cdot_\g y-y\cdot_\g x$ defines a Lie algebra structure
on $\g$, which is called the {\bf sub-adjacent Lie algebra} of
$(\g,\cdot_\g)$ and denoted by $\g^c$. Furthermore,
$L:\g\longrightarrow \gl(\g)$ with $x\rightarrow L_x$, where
$L_xy=x\cdot_\g y$, for all $x,y\in \g$, gives a representation of
the Lie algebra $\g^c$ on $\g$.

\begin{defi}{\rm (\cite{Left-symmetric bialgebras})}
Let $(\g,\cdot_\g)$ be a pre-Lie algebra and $V$  a vector
space. A {\bf representation} of $\g$ on $V$ consists of a pair
$(\rho,\mu)$, where $\rho:\g\longrightarrow \gl(V)$ is a representation
of the Lie algebra $\g^c$ on $V $ and $\mu:\g\longrightarrow \gl(V)$ is a linear
map satisfying \begin{eqnarray}\label{representation condition 2}
 \rho(x)\mu(y)u-\mu(y)\rho(x)u=\mu(x\cdot_\g y)u-\mu(y)\mu(x)u, \quad \forall~x,y\in \g,~ u\in V.
\end{eqnarray}
\end{defi}

Usually, we denote a representation by $(V;\rho,\mu)$. It is
obvious that $(  \K  ;\rho=0,\mu=0)$ is a representation, which we
call the {\bf trivial representation}. Let $R:\g\rightarrow
\gl(\g)$ be a linear map with $x\longrightarrow R_x$, where the
linear map $R_x:\g\longrightarrow\g$  is defined by
$R_x(y)=y\cdot_\g x,$ for all $x, y\in \g$. Then
$(\g;\rho=L,\mu=R)$ is also a representation, which we call the
{\bf regular representation}. Define two linear maps $L^*,R^*:\g\longrightarrow
\gl(\g^*)$   with $x\longrightarrow L^*_x$ and
$x\longrightarrow R^*_x$ respectively (for all $x\in \g$)
by
\begin{equation}
\langle L_x^*(\alpha),y\rangle=-\langle \alpha, x\cdot_\g y\rangle, \;\;
\langle R_x^*(\alpha),y\rangle=-\langle \alpha, y\cdot_\g x\rangle, \;\;
\forall x, y\in \g, \alpha\in \g^*.
\end{equation}
Then $(\g^*;\rho={\rm ad}^*=L^*-R^*, \mu=-R^*)$ is a
representation of $(\g,\cdot_\g)$. In fact, it is the dual
representation of the regular representation $(\g;L,R)$ and is called {\bf coregular representation}.

The cohomology complex for a pre-Lie algebra $(\g,\cdot_\g)$ with a representation $(V;\rho,\mu)$ is given as follows (\cite{Burde}).
The set of $n$-cochains is given by
$\Hom(\wedge^{n-1}\g\otimes \g,V),\
n\geq 1.$  For all $\phi\in \Hom(\wedge^{n-1}\g\otimes \g,V)$, the coboundary operator $\dM:\Hom(\wedge^{n-1}\g\otimes \g,V)\longrightarrow \Hom(\wedge^{n}\g\otimes \g,V)$ is given by
 \begin{eqnarray}\label{eq:pre-Lie cohomology}
 \dM\phi(x_1, \cdots,x_{n+1})
 \nonumber&=&\sum_{i=1}^{n}(-1)^{i+1}\rho(x_i)\phi(x_1, \cdots,\hat{x_i},\cdots,x_{n+1})\\
\nonumber &&+\sum_{i=1}^{n}(-1)^{i+1}\mu(x_{n+1})\phi(x_1, \cdots,\hat{x_i},\cdots,x_n,x_i)\\
 \nonumber&&-\sum_{i=1}^{n}(-1)^{i+1}\phi(x_1, \cdots,\hat{x_i},\cdots,x_n,x_i\cdot_\g x_{n+1})\\
\label{eq:cobold} &&+\sum_{1\leq i<j\leq n}(-1)^{i+j}\phi([x_i,x_j]_\g,x_1,\cdots,\hat{x_i},\cdots,\hat{x_j},\cdots,x_{n+1}),
\end{eqnarray}
for all $x_i\in \g,~i=1,\cdots,n+1$. \emptycomment{In particular, we use the
symbol $\dM^T$ to refer the coboundary operator   associated to
the trivial representation and $\dr$ to refer the coboundary
operator  associated to the regular representation. We denote the
$n$-th cohomology group for the coboundary operator $\dr$  by
$H_{\rm reg}^n(\g,\g)$ and $H_{\rm reg}(\g,\g)=\oplus_{n}H_{\rm reg}^n(\g,\g)$.}

\subsection{Relative Rota-Baxter operators on Lie algebras and their Maurer-Cartan characterizations }
\begin{defi}{\rm(\cite{Ku99b})}
Let $(\g,[\cdot,\cdot]_\g)$ be a Lie algebra and $\rho:\g\longrightarrow\gl(V)$ a representation of $\g$ on a vector space $V$. A {\bf relative Rota-Baxter operator} (or $\huaO$-operator) on $\g$ with respect to the representation $(V;\rho)$ is a linear map $T:V\longrightarrow\g$ such that
 \begin{equation}
   [Tu,Tv]_\g=T\big(\rho(Tu)(v)-\rho(Tv)(u)\big),\quad \forall~u,v\in V.
 \label{eq:defiO}
 \end{equation}
\end{defi}

\begin{defi}
    Let $T_\g$ and $T_\h$ be relative Rota-Baxter operators on Lie algebras $\g$ and $\h$ associated to representations $(V_\g;\rho_\g)$ and $(V_\h;\rho_\h)$
    respectively. A {\bf homomorphism of relative Rota-Baxter operators} from $T_\g$ to $T_\h$ consists
    of a Lie algebra homomorphism  $\phi:\g\rightarrow\h$ and a linear map $\alpha:V_\g\rightarrow V_\h$ such that for all $x\in\g, v\in V_\g$,
    \begin{eqnarray}
        \alpha\rho_\g(x)(v)&=&\rho_\h(\phi(x))(\alpha(v)),\label{defi:isocon2b}\\
        T_\h\circ \alpha &=&\phi\circ T_\g.\label{defi:isocon1b}
    \end{eqnarray}
    In particular, if $\phi$ and $\alpha$ are  invertible,  then $(\phi,\alpha)$ is called an  {\bf isomorphism}  from $T_\g$ to
    $T_\h$. Let ${\bf RB}$ denote the category of
    relative Rota-Baxter operators with the above morphisms.
\end{defi}

\emptycomment{\begin{defi}
Let $T_1$ and $T_2$ be two relative Rota-Baxter operators on a Lie algebra $\g$ with respect to the representation $(V;\rho)$. A {\bf homomorphism} from $T_2$ to $T_1$ consists of a Lie algebra homomorphism $\phi:\g\rightarrow \g$ and a linear map $\psi:V\rightarrow V$ satisfying
\begin{eqnarray}
  T_1\circ \psi &=& \phi\circ T_2, \\
  \psi(\rho(x)v) &=& \rho(\phi(x))\psi(v),\quad x\in\g,v\in V.
\end{eqnarray}
Furthermore, if $\phi$ and $\psi$ are linear isomorphism, $(\phi,\psi)$ is called an {\bf isomorphism} from $T_2$ to $T_1$.
\end{defi}}

\begin{thm}{\rm(\cite{Bai-1})}
Let $T:V\to \g$ be a relative Rota-Baxter operator on a Lie algebra $(\g,[\cdot,\cdot]_\g)$ with respect to a representation $(V;\rho)$. Define a multiplication $\cdot^T$ on $V$ by
\begin{equation}
  u\cdot^T v=\rho(Tu)(v),\quad \forall u,v\in V.
\end{equation}
Then $(V,\cdot^T)$ is a pre-Lie algebra.
 \end{thm}

We denote by $(V,[\cdot,\cdot]^T)$ the sub-adjacent Lie algebra of the pre-Lie algebra $(V,\cdot^T)$. More precisely,
\begin{equation}\label{eq:bracketT}
 ~[u,v]^T=\rho(Tu)(v)-\rho(Tv)(u).
\end{equation}
Moreover,  $T$ is a Lie algebra homomorphism from $(V,[\cdot,\cdot]^T)$ to $(\g,[\cdot,\cdot]_\g)$.

\begin{defi}
  Let $(\g=\oplus_{k=0}^\infty\g_i,[\cdot,\cdot],\dM)$ be a differential graded Lie algebra.  A degree $1$ element $\theta\in\g_1$ is called a {\bf Maurer-Cartan element} of $\g$ if it
  satisfies the following {\bf Maurer-Cartan equation}:
  \begin{equation}
  \dM \theta+\half[\theta,\theta]=0.
  \label{eq:mce}
  \end{equation}
  \end{defi}
A graded Lie algebra is a differential graded Lie
algebra with $d=0$.

Let $(V;\rho)$ be a representation of a Lie algebra $\g$. Consider the graded vector space
$$\huaC^*(V,\g):=\oplus_{k=0}^{\dim(V)}\Hom(\wedge^{k}V,\g).$$
Define a skew-symmetric bracket operation $$\Courant{\cdot,\cdot}: \Hom(\wedge^nV,\g)\times \Hom(\wedge^mV,\g)\longrightarrow \Hom(\wedge^{m+n}V,\g)$$ by
\begin{eqnarray}
&&\nonumber\Courant{P,Q}(u_1,u_2,\cdots,u_{m+n})\\
\label{o-bracket}&=&\sum_{\sigma\in \mathbb S_{(m,1,n-1)}}(-1)^{\sigma}P(\rho(Q(u_{\sigma(1)},\cdots,u_{\sigma(m)}))u_{\sigma(m+1)},u_{\sigma(m+2)},\cdots,u_{\sigma(m+n)})\\
\nonumber&&-(-1)^{mn}\sum_{\sigma\in \mathbb S_{(n,1,m-1)}}(-1)^{\sigma}Q(\rho(P(u_{\sigma(1)},\cdots,u_{\sigma(n)}))u_{\sigma(n+1)},u_{\sigma(n+2)},\cdots,u_{\sigma(m+n)})\\
\nonumber&&+(-1)^{mn}\sum_{\sigma\in \mathbb S_{(n,m)}}(-1)^{\sigma}[P(u_{\sigma(1)},\cdots,u_{\sigma(n)}),Q(u_{\sigma(n+1)},\cdots,u_{\sigma(m+n)})]
\end{eqnarray}
for all $P\in\Hom(\wedge^nV,\g)$ and $Q\in\Hom(\wedge^mV,\g)$.

Furthermore, one has
\begin{thm}\label{pro:gla}{\rm(\cite{TBGS})}
 $(\huaC^*(V,\g),\Courant{\cdot,\cdot})$ is a graded Lie algebra. Its Maurer-Cartan elements are precisely the relative Rota-Baxter operators on $\g$ with respect to the representation $(V;\rho)$.
\end{thm}

In the following, we recall cohomology for relative Rota-Baxter operators, which corresponds to the
Chevalley-Eilenberg cohomology.
\begin{lem}\label{lem:rep}{\rm(\cite{TBGS})}
Let $T$ be a relative Rota-Baxter operator on a Lie
algebra $\g$ with respect to a representation $(V;\rho)$.
  Define
 \begin{equation}
\varrho:V\longrightarrow\gl(\g), \quad     \varrho(u)(x):=[Tu,x]_\g+T\rho(x)(u),\;\;\forall x\in \g,u\in
    V.
  \end{equation}
Then  $\varrho$ is a representation of the sub-adjacent Lie algebra $(V,[\cdot,\cdot]^T)$ on the vector space $\g$.
  \end{lem}

Let $\delta_{\rm RB}: \Hom(\wedge^kV,\g)\longrightarrow
\Hom(\wedge^{k+1}V,\g)$ be the corresponding Chevalley-Eilenberg
coboundary operator. More precisely, for all $f\in
\Hom(\wedge^kV,\g)$ and $u_1,\cdots,u_{k+1}\in V$, we have
\begin{eqnarray}
&& \delta_{\rm RB} f(u_1,\cdots,u_{k+1})\notag\\
&=&\sum_{i=1}^{k+1}(-1)^{i+1}[Tu_i,f(u_1,\cdots,\hat{u_i},\cdots, u_{k+1})]_\g+\sum_{i=1}^{k+1}(-1)^{i+1}T\rho(f(u_1,\cdots,\hat{u_i},\cdots, u_{k+1}))(u_i)\notag\\
&&+\sum_{1\le i<j\le k+1}(-1)^{i+j}f(\rho(Tu_i)(u_j)-\rho(Tu_j)(u_i),u_1,\cdots,\hat{u_i},\cdots,\hat{u_j},\cdots, u_{k+1}).\label{eq:odiff}
\end{eqnarray}

The coboundary operators $\delta_{\rm RB}$ can also be given in terms of the graded Lie algebra $(\huaC^*(V,\g),\Courant{\cdot,\cdot})$ by the following relation:
\begin{pro}\label{pro:danddT}{\rm(\cite{TBGS})}
 Let $T:V\longrightarrow\g$ be a relative Rota-Baxter operator on a Lie algebra $\g$ with respect to a representation $(V;\rho)$. Then we have
 $$
 \delta_{\rm RB} f=(-1)^k\Courant{T,f},\quad  f\in \Hom(\wedge^kV,\g).
 $$
\end{pro}

\section{Endo pre-Lie algebras, Manin triples and endo symplectic Lie algebras}\label{sec:endoprelie}
In this section we introduce the notion of endo pre-Lie algebras and give the equivalent structures of pre-Lie bialgebras, matched pairs, Manin triples for endo pre-Lie algebras and phase spaces of Lie algebras.
\subsection{ Endo pre-Lie algebras and their representations}
\begin{defi}
An {\bf endo pre-Lie algebra} is a triple $(\g,\cdot_\g,\phi)$, or simply $(\g, \phi)$, where $(\g, \cdot_\g)$ is a pre-Lie algebra and $\phi:\g\rightarrow\g$ is a pre-Lie algebra endomorphism.
\end{defi}
As an analogy of a pre-Lie algebra representation, we have
\begin{defi}
A {\bf representation} of an endo pre-Lie algebra $(\g, \phi)$ is $(V; \rho, \mu, \alpha)$,  where $(V; \rho, \mu)$ is a representation of a pre-Lie algebra $\g$ and $\alpha\in \End(V)$ such that
\begin{eqnarray}
\label{representation condition 1}\alpha(\rho(x)(v))&=&\rho(\phi(x))(\alpha(v))\\
 \label{representation condition 2}\alpha(\mu(x)(v))&=&\mu(\phi(x))(\alpha(v)), \quad \forall~x\in \g, ~ v\in V.
\end{eqnarray}
 \end{defi}
Two representations $(V_1;\rho_1, \mu_1, \alpha_1)$ and $(V_2;\rho_2, \mu_2, \alpha_2)$ of an endo pre-Lie algebra $(\g, \phi)$ are called {\bf equvialent} if exists a linear isomorphism $\varphi:V_1\rightarrow V_2$ such that
\begin{eqnarray}
 \varphi(\rho_1(x)(v))=\rho_2(x)( \varphi(v)) ,  \varphi(\mu_1(x)(v))=\mu_2(x)( \varphi(v)), \quad \forall~x\in \g, ~ v\in V.
\end{eqnarray}
For vector spaces $ V_1$and $ V_2$, and linear maps $\phi_1: V_1\rightarrow V_1$ and $\phi_2: V_2\rightarrow V_2$,  we denote $\phi_1+\phi_2$ by the linear map
\begin{eqnarray}
\phi_{ V_1\oplus V_2}: V_1\oplus V_2\rightarrow V_1\oplus V_2,\quad \phi_{ V_1\oplus V_2}(v_1+v_2):=\phi_1(v_1)+\phi_2(v_2).
\end{eqnarray}
For a pre-Lie algebra $(\g, \cdot_\g)$ , a linear space $ V$, two linear maps $\rho, \mu:\g\rightarrow\gl(V)$, define a multiplication $\ast$ on $\g\oplus V$ by
\begin{eqnarray}
(x_1+v_1)\ast(x_2+v_2)=x_1{\cdot_\g}x_2+\rho(x_1)v_2+\mu(x_2)v_1.\quad \forall~x_1, x_2\in \g, ~ v_1, v_2\in V.
\end{eqnarray}
\begin{lem}{\rm(\cite{Left-symmetric bialgebras})}\label{lem:semi-direct}
$(\g\oplus V,\ast)$ is a pre-Lie algebra if and only if $(V;\rho, \mu)$ is a representation of a pre-Lie algebra $\g$.
\end{lem}
The pre-Lie algebra is called the {\bf semi-direct product} of $\g$ and $ V$ , denoted $\g\ltimes_{\rho, \mu} V$ or $\g\ltimes V$.
\begin{pro}
Let $(\g, \phi)$ be an endo pre-Lie algebra , $(V;\rho, \mu)$ is a representation of a pre-Lie algebra $\g$ and $\alpha\in \End(V)$ .Then $(\g\ltimes V, \phi+\alpha)$ is a endo pre-Lie algebra if and only if $(V;\rho, \mu, \alpha)$ is a representation of $(\g, \phi)$.
\end{pro}
\begin{proof}
Note that $(\g\ltimes V, \phi+\alpha)$ is a endo pre-Lie algebra if and only if $ \g\ltimes V$ is a pre-Lie algebra and $\phi+\alpha\in \End(\g\ltimes V)$ is a pre-Lie algebra homomorphism. By Lemma \ref{lem:semi-direct}, we only need to prove that $\phi+\alpha\in \End(\g\ltimes V)$ is a pre-Lie algebra homomorphism. In fact, we have
\begin{eqnarray*}
(\phi+\alpha)((x_1+v_1)\ast(x_2+v_2))&=&(\phi+\alpha)(x_1{\cdot_\g}x_2+\rho(x_1)v_2+\mu(x_2)v_1)\\
&=&\phi(x_1{\cdot_\g}x_2)+\alpha(\rho(x_1)v_2+\mu(x_2)v_1),\\
(\phi+\alpha)(x_1+v_1)\ast(\phi+\alpha)(x_2+v_2)&=&(\phi(x_1)+\alpha(v_1))\ast(\phi(x_2)+\alpha(v_2))\\
&=&\phi(x_1)\cdot_\g\phi(x_2)+\rho(\phi(x_1))\alpha(v_2)+\mu(\phi(x_2))\alpha(v_1),
\end{eqnarray*}
which implies that $\phi+\alpha\in \End(\g\ltimes V)$ is a pre-Lie algebra homomorphism .
\end{proof}

\begin{lem}\label{lem:dual representation}
Let $(\g, \phi)$ be an endo pre-Lie algebra. Let $(V;\rho, \mu)$ be a representation of $\g$. For $\beta\in \End( V)$, $( V^\ast; \rho^\ast-\mu^\ast, -\mu^\ast, \beta^\ast)$ is a representation of  $(\g,\phi)$ if and only if
\begin{eqnarray}
\beta((\rho-\mu)(\phi(x))(v))=(\rho-\mu)(x)(\beta(v)),\quad
\beta(\mu(\phi(x))(v))=\mu(x)(\beta(v)),\quad x\in \g, ~ v\in V.
\end{eqnarray}
\end{lem}
\begin{proof}
It follows by a direct calculation.
\end{proof}
\begin{defi}
Let $(\g, \phi)$ be an endo pre-Lie algebra, $(V;\rho, \mu)$ is a representation of a pre-Lie algebra $\g$ and $\beta\in \End(V)$. We say that $\beta$ {\bf dually represents endo pre-Lie algebra} $(\g, \phi)$ if $( V^\ast; \rho^\ast-\mu^\ast, -\mu^\ast, \beta^\ast)$ is a representation of  $(\g, \phi)$.
\end{defi}

\begin{ex}
Let $(\g, \phi)$ be an endo pre-Lie algebra. Then $(\g;L,R,\phi)$ is naturally a representation of the endo pre-Lie algebra $(\g, \phi)$, called the {\rm regular representation}. Furthermore, a linear operator $\psi$ on $\g$ dually represents $(\g, \phi)$ if and only if for  any $x,y\in\g$,  any two of the following  three equations hold:
\begin{eqnarray}
\psi([\phi(x),  y]_\g)=[x, \psi(y)]_\g,\quad \psi(\phi(x)\cdot_\g  y)=x\cdot_\g \psi(y), \quad \psi(x\cdot_\g \phi(y))=\psi(x)\cdot_\g y.
\end{eqnarray}
\end{ex}

By Lemma \ref{lem:dual representation}, we have
\begin{cor}
Let $(\g, \phi)$ be an endo pre-Lie algebra, $(V;\rho, \mu)$ is a representation of a pre-Lie algebra $\g$ and $\beta\in \End(V)$. If $\beta$ dually represents endo pre-Lie algebra $(\g, \phi)$, then we have the semi-direct product endo pre-Lie algebra $(\g\ltimes V^\ast, \phi+\beta^\ast)$.
\end{cor}
\subsection{Matched pairs of endo pre-Lie algebras}
\begin{defi}{\rm (\cite{Left-symmetric bialgebras})}
{\bf A matched pair of pre-Lie algebra} is $(\g, \h, \rho_\g,\mu_\g, \rho_\h, \mu_\h)$, where $(\g, \cdot_\g)$ and $(\h, \cdot_\h)$ are pre-Lie algebras, $(\h;\rho_\g, \mu_\g)$ is a representation of the pre-Lie algebra $\g$ and $(\g;\rho_\h, \mu_\h)$ is a representation of the pre-Lie algebra $\h$, such that for $x, y\in\g, a, b\in\h$, the following equalities hold:
\begin{eqnarray*}
\label{representation condition 1}\mu_\g(x)[a, b]_\h&=&\mu_\g(\rho_\h(b)x)a-\mu_\g(\rho_\h(a)x)b+a\cdot_\h(\mu_\g(x)b)-b\cdot_\h(\mu_\g(x)a),\\
\label{representation condition 2}\rho_\g(x)(a\cdot_\h b)&=&-\rho_g(\rho_\h(a)x-\mu_\h(a)x)b+(\rho_\g(x)a-\mu_\g(x))\cdot_\h b+\mu_\g(\mu_\h(b)x)a+a\cdot_\h(\rho_\g(x)b),\\
\label{representation condition 3}\mu_\h(a)[x, y]_\g&=&\mu_\h(\rho_\g(y)a)x-\mu_\h(\rho_\g(x)a)y+x\cdot_\g(\mu_\h(a)y)-y\cdot_\g(\mu_\h(a)x),\\
\label{representation condition 4}\rho_\h(a)(x\cdot_\g y)&=&-\rho_\h(\rho_\g(x)a-\mu_\g(x)a)y+(\rho_\h(a)x-\mu_\h(a)x)\cdot_\g y+\mu_\h(\mu_\g(y)a)x+x\cdot_\g(\rho_\h(a)y).
\end{eqnarray*}
\end{defi}

\begin{thm}{\rm (\cite{Left-symmetric bialgebras})}\label{thm:matched pairs of pre-Lie}
Let $(\g, \cdot_\g)$ and $(\h, \cdot_\h)$ be pre-Lie algebras, $(\h;\rho_\g, \mu_\g)$ be a representation of the pre-Lie algebra $\g$ and $(\g;\rho_\h, \mu_\h)$ be a representation of the pre-Lie algebra $\h$. Define a multiplication on the direct sum $\g\oplus\h$ by
\begin{eqnarray}
\nonumber (x+a)\ast(y+b)&=&x\cdot_\g y+\rho_\h(a)y+\mu_\h(b)x\\
&&+a\cdot_\h b+\rho_\g(x)b+\mu_\g(y)a ,\quad \forall x, y\in \g, a, b\in \h.
\end{eqnarray}
Then $(\g\oplus\h, \ast)$ is a pre-Lie algebra if and only if $(\g, \h, \rho_\g, \mu_\g, \rho_\h, \mu_\h)$ is a matched pair of pre-Lie algebras of $(\g, \cdot_\g)$ and $(\h, \cdot_\h)$. We denote the resulting pre-Lie algebra $(\g\oplus\h, \ast)$ by $\g\bowtie\h$.
\end{thm}

\begin{defi}\label{def:matched pair of endo pre-Lie algebras}
A {\bf matched pair of endo pre-Lie algebras} is $((\g, \phi_\g), (\h, \phi_\h), \rho_\g, \mu_\g, \rho_\h, \mu_\h)$, where $(\g, \phi_\g), (\h, \phi_\h)$ are endo pre-Lie algebras. $\rho_\g, \mu_\g:\g\rightarrow \gl(\h), \rho_\h, \mu_\h:\h\rightarrow \gl(\g)$ are linear maps, such that
\begin{itemize}
\item[\rm(i)]$(\g;\rho_\h, \mu_\h, \phi_\g)$ is a representation of endo pre-Lie algebra $(\h, \phi_\h), $
\item[\rm(ii)] $(\h; \rho_\g, \mu_\g, \phi_\h)$ is a representation of endo pre-Lie algebra $(\g, \phi_\g), $
\item[\rm(iii)]$(\g, \h, \rho_\g, \mu_\g, \rho_\h, \mu_\h)$ is a matched pair of pre-Lie algebras.
\end{itemize}
\end{defi}

The endo pre-Lie algebras have the following characterization.
\begin{thm}\label{thm:endo pre-Lie algebra}
Let  $(\g, \phi_\g)$ and $(\h, \phi_\h)$ be endo pre-Lie algebras and $(\g, \h, \rho_\g, \mu_\g, \rho_\h, \mu_\h)$ be a matched pair of pre-Lie algebra of $(\g, \cdot_\g)$ and $(\h, \cdot_\h)$. Then $(\g\bowtie \h, \phi_\g+\phi_\h)$ is an endo pre-Lie algebra if and only if $((\g, \phi_\g), (\h, \phi_\h), \rho_\g, \mu_\g, \rho_\h, \mu_\h)$ is a matched pair of endo pre-Lie algebras $(\g, \phi_\g)$ and $(\h, \phi_\h)$.
\end{thm}
\begin{proof}
By Theorem \ref{thm:matched pairs of pre-Lie}, $\g\bowtie h$ is a pre-Lie algebra if and only if $(\g, \h, \rho_\g, \mu_\g, \rho_\h, \mu_\h)$ is a matched pair of pre-Lie algebras of $(\g, \cdot_\g)$ and $(\h, \cdot_\h)$. Therefore, in order
to prove that $(\g\bowtie \h, \phi_\g+\phi_\h)$ is a endo pre-Lie algebra, we just need to  prove $\phi_\g+\phi_\h$ is a pre-Lie algebra homomorphism. By a direct calculation, we have
\begin{eqnarray*}
(\phi_\g+\phi_\h)((x+a)\ast(y+b))&=&(\phi_\g+\phi_\h)(x\cdot_\g y+\rho_\h(a)y+\mu_\h(b)x+a\cdot_\h b+\rho_\g(x)b+\mu_\g(y)a)\\
&=&\phi_\g(x\cdot_\g y+\rho_\h(a)y+\mu_\h(b)x)+\phi_\h(a\cdot_\h b+\rho_\g(x)b+\mu_\g(y)a),\\
(\phi_\g+\phi_\h)(x+a)\ast(\phi_\g+\phi_\h)(y+b)&=&(\phi_\g(x)+\phi_\h(a))\ast(\phi_\g(y)+\phi_\g(x))\\
&=&\phi_\g(x)\cdot_\g\phi_\g(y)+\rho_\h(\phi_\h(a))\phi_\g(y)+\mu_\h(\phi_\h(b))\phi_\g(x)\\
&&+\phi_\h(a)\cdot_\h\phi_\h(b)+\rho_\g(\phi_\g(x))\phi_\h(b)+\mu_\g(\phi_\g(y))\phi_\h(a).
\end{eqnarray*}
If $\phi_\g+\phi_\h$ is a pre-Lie algebra homomorphism, letting $x=b=0$, we have
\begin{eqnarray*}
\phi_\g(\rho_\h(a)y)+\phi_\h(\mu_\g(y)a)=\rho_\h(\phi_\h(a))\phi_\g(y)+\mu_\g(\phi_\g(y))\phi_\h(a),
\end{eqnarray*}
which implies that
$$\phi_\g(\rho_\h(a)y)=\rho_\h(\phi_\h(a))\phi_\g(y),\quad \phi_\h(\mu_\g(y)a)=\mu_\g(\phi_\g(y))\phi_\h(a).$$
Similarly, for $y=a=0$, we have
\begin{eqnarray*}
\phi_\g(\mu_\h(b)x)=\mu_\h(\phi_\h(b))\phi_\g(x),\quad \phi_\h(\rho_\g(x)b)=\rho_\g(\phi_\g(x))\phi_\h(b).
\end{eqnarray*}
Thus, $(\g; \rho_\h, \mu_\h, \phi_\g)$ is a representation of the endo pre-Lie algebra $(\h, \phi_\h)$ and $(\h; \rho_\g, \mu_\g, \phi_\h)$ is a representation of the endo pre-Lie algebra $(\g, \phi_\g)$. The converse can be proved similarly.
\end{proof}
\subsection{ Endo symplectic Lie algebras and Manin triples of endo pre-Lie algebras}

A nondegenerate  skew-symmetric bilinear form
$(\cdot,\cdot)_-$ on a pre-Lie algebra $(\g,\cdot_\g)$
is called {\bf invariant} if
\begin{equation}
	(x\cdot_\g y,z)_-+(y,[x,z]_\g)_-=0,\quad\forall~x,y,z\in\g.
\end{equation}

\begin{defi}
	A {\bf quadratic pre-Lie algebra} is a triple $(\g,\cdot_\g,(\cdot,\cdot)_-)$, where $(\g,\cdot_\g)$ is a pre-Lie algebra and $(\cdot,\cdot)_-$ is a nondegenerate invariant skew-symmetric bilinear form on $\g$.
\end{defi}

\begin{ex}{\rm
		Let $(\g,\cdot_\g)$ be a pre-Lie algebra. Then $(\frkd=\g\ltimes_{\ad^*,-R^*} \g^*,(\cdot,\cdot)_-)$ is a quadratic pre-Lie algebra, where the nondegenerate invariant skew-symmetric bilinear form $(\cdot,\cdot)_-$ is given by
		\begin{eqnarray}\label{symplectic bracket}
			(x+a,y+b)_-=\langle a,y\rangle-\langle x,b\rangle, \quad \forall ~x,y\in \g,~~a,b\in\g^*.
	\end{eqnarray}}
\end{ex}

Recall that a nondegenerate skew-symmetric bilinear form $\omega\in\wedge^2\g^*$ is called a {\bf symplectic structure} on a Lie algebra $(\g,[\cdot,\cdot]_\g )$ if
$$
\omega([x,y]_\g,z)+\omega([y,z]_\g,x)+\omega([z,x]_\g,y)=0,\quad  \forall x,y,z \in \g.
$$
A symplectic Lie algebra is a Lie algebra equipped with a symplectic structure.

There is a one-to-one correspondence between quadratic pre-Lie algebras and symplectic Lie algebras.
\begin{thm}{\rm(\cite{Chu})}\label{lem:sym-pre}
	Let $(\g,  \cdot_\g,\omega)$ be a quadratic pre-Lie algebra. Then $(\g,[\cdot,\cdot]_\g,\omega)$ is a symplectic Lie algebra. Conversely, let $(\g,[\cdot,\cdot]_\g,\omega)$ be a symplectic Lie algebra. Then there exists a compatible pre-Lie algebra structure $\cdot_\g$ on $\g$ given by
	\begin{equation}\label{eq:pre-Lie-symplectic}
		\omega(x\cdot_\g y,z)=-\omega(y,[x,z]_\g),\quad~\forall x,y,z\in\g,
	\end{equation}
	whose sub-adjacent Lie algebra is exactly $(\g,[\cdot,\cdot]_\g)$. Moreover, $(\g,\cdot_\g,\omega)$ is a quadratic pre-Lie algebra.
\end{thm}

\begin{defi}
	A   {\bf Manin triple of pre-Lie algebras}
	is a triple $((\frkd,\cdot_\frkd, (\cdot,\cdot)_-),\g_1,\g_2)$, where $(\frkd,\cdot_\frkd,(\cdot,\cdot)_-)$ is an even dimensional quadratic pre-Lie algebra, $\g_1$ and $\g_2$ are pre-Lie subalgebras, both isotropic with respect to $(\cdot,\cdot)_-$, and $\frkd=\g_1\oplus\g_2$ as vector spaces.
\end{defi}

\begin{defi}{\rm(\cite{Ku94})}
	Let $\g$ be a Lie algebra.  If there exists a Lie algebra structure on a direct sum of the underlying vector space of $\g$ and $\g^*$ such that $\g$ and $\g^*$ are Lie subalgebras and the natural skew-symmetric bilinear form on $\g\oplus \g^*$
	\begin{equation}
		\omega_p(x+a,y+b)=\langle a,y\rangle-\langle x,b\rangle,\quad\forall~x,y\in\g,a,b\in\g^*
	\end{equation}
	is a $2$-cocycle, then it is called a {\bf phase space} of the Lie algebra $\g$.
\end{defi}

\begin{thm}{\rm (\cite{Left-symmetric bialgebras})}\label{thm:Manin-Mathched pair}
Let $(\g, \cdot_\g)$ be a pre-Lie algebra. Suppose that there is a pre-Lie algebra structure $\cdot_{\g^\ast}$ on its dual space $\g^\ast$. Then the following statements are equivalent:
\begin{itemize}
  \item[{\rm(i)}] $(\frkd=\g\oplus \g^\ast, \g, \g^\ast,(\cdot,\cdot)_-=\omega_p)$ is a Manin triple;
  \item[{\rm(ii)}] $(\g, \g^\ast, \ad^\ast, -R^\ast, \AD^\ast, -\frkR^\ast)$ is a matched pair of pre-Lie algebras;
  \item[{\rm(iii)}]$\omega_p$ is a phase space of the sub-adjacent Lie algebra $\g^c$, where 
       the Lie bracket $[\cdot,\cdot]_p$ on $\g\oplus \g^*$ is defined by
    \begin{eqnarray}\label{eq:phase bracket}
    [x+a,y+b]_p=[a,b]_{\g^*}+L^*_xb-L^*_ya+\frkL^*_a y-\frkL^*_b x+[x,y]_\g,\quad \forall~x,y\in\g,a,b\in\g^*.
    \end{eqnarray}
       Furthermore, every phase space of a Lie algebra can be obtained in this way.
  \end{itemize}
\end{thm}

\begin{defi}
 A {\bf quadratic  endo  pre-Lie algebra} is $(\frkd, \phi, (\cdot,\cdot)_-)$, where $(\frkd, \phi)$ is an endo pre-Lie algebra and $(\frkd, (\cdot,\cdot)_-)$ is a quadratic pre-Lie algebra.
\end{defi}

Let $(\g,\omega)$ be a symplectic Lie algebra and $\phi:\g\rightarrow\g $ a linear map. Let $\widehat{\phi} :\g\rightarrow\g$ denote the {\bf adjoint linear transformation} of $\phi$ with respective to $\omega$ by
\begin{eqnarray}
\omega(\phi(x), y)=\omega(x, \widehat{\phi}(y)),\quad \forall~ x, y\in \g.
\end{eqnarray}
\begin{defi}
Let $(\g,\omega)$ be a symplectic Lie algebra and $\phi:\g\rightarrow\g $ a linear map. If the { adjoint linear transformation} $\widehat{\phi} $ of $\phi$ satisfies
\begin{equation}\label{eq:endo sym1}
\widehat{\phi}([\phi(x), y]_\g)=[x, \widehat{\phi}(y)]_\g,\quad \forall~x,y\in\g,
\end{equation}
then $(\g,\phi,\omega)$ is called an {\bf endo symplectic Lie algebra}.
\end{defi}

\begin{pro}\label{pro:quadratic endo endo symp}
There is a one-to-one correspondence between quadratic endo pre-Lie algebras and  endo symplectic Lie algebras.
\end{pro}
\begin{proof}
For $x,y,z\in\g$, we have
\begin{eqnarray*}
\omega(\phi(x\cdot_\g y)-\phi(x)\cdot_\g\phi(y), z)&=&\omega(x\cdot_\g y, \widehat{\phi}(z))+\omega(\phi(y), [\phi(x), z]_\g)\\
&=&\omega(y, \widehat{\phi}[\phi(x), z]_\g-[x, \widehat{\phi}(z)]_\g),
\end{eqnarray*}
which implies that $\phi$ is a pre-Lie algebra homomorphism if and only if \eqref{eq:endo sym1} holds.

By Theorem \ref{lem:sym-pre}, there is a one-to-one correspondence between quadratic pre-Lie algebras and symplectic Lie algebras. The conclusion follows.
\end{proof}

Let $(\frkd, \phi, (\cdot,\cdot)_-)$ be a { quadratic  endo  pre-Lie algebra}. We also let $\widehat{\phi} :\g\rightarrow\g$ denote the { adjoint linear transformation} of $\phi$ under the nondegenerate skew-symmetric  bilinear form $(\cdot,\cdot)_-$.
\begin{eqnarray}
(\phi(x), y)_-=(x, \widehat{\phi}(y))_-,\quad \forall x, y\in \g.
\end{eqnarray}
\begin{pro}
Let $(\frkd, \phi, (\cdot,\cdot)_-)$ be a quadratic  endo  pre-Lie algebra and  $\widehat{\phi}$ be the adjoint of $\phi$. Then $\widehat{\phi}$ dually represents $(\frkd, \phi)$, that is, $(\g^*;\ad^*,-R^*,\widehat{\phi}^*)$ is a representation of the endo pre-Lie algebra $(\frkd, \phi)$.
\end{pro}
\begin{proof}
	For $x,y,z\in\frkd$, we have
\begin{eqnarray*}
0 &=&(\phi(x\cdot_\frkd y), z)_--(\phi(x)\cdot_\frkd\phi(y), z)_-\\
&=&(x\cdot_\frkd y, \widehat{\phi}(z))_-+(\phi(y), [\phi(x), z]_\frkd)_-\\
&=&-(y, [x, \widehat{\phi}(z)]_\frkd )_-+(y, \widehat{\phi}([\phi(x), z]_\frkd))_-,
\end{eqnarray*}
which implies that
\begin{equation*}\label{eq:bracket rep}
[x, \widehat{\phi}(z)]_\frkd=\widehat{\phi}([\phi(x), z]_\frkd).
\end{equation*}

Similarly, by fact $(\phi([x, y]_\frkd, z)_--([\phi(x),\phi(y)]_\frkd, z)_-=0$, we have
\begin{equation*}
\widehat{\phi}(z)\cdot_\frkd x=\widehat{\phi}(z\cdot_\frkd \phi(x)).
\end{equation*}
Thus $\widehat{\phi}$ dually represents $(\frkd, \phi)$.
\end{proof}

\begin{defi}
Let $(\g, \phi)$ be an endo pre-Lie algebra. Suppose that $(\g^\ast, \psi^\ast)$ is also an endo pre-Lie algebra. {\bf A Manin triple of endo pre-Lie algebras} is  $((\frkd=\g\oplus \g^\ast, \phi+\psi^\ast,(\cdot,\cdot)_-), (\g, \phi), (\g^\ast, \psi^\ast))$ associated to $(\g, \phi)$ and $(\g^\ast, \psi^\ast)$  such that $(\frkd, \phi+\psi^\ast,  (\cdot,\cdot)_-) $ is a quadratic endo pre-Lie algebra and $(\frkd, \g, \g^\ast)$ is a Manin triple of pre-Lie algebras associated to $\g$ and $\g^\ast$.
\end{defi}
\begin{lem}
Let $(\g\bowtie \g^\ast, \phi+\psi^\ast, (\cdot,\cdot)_-) $ be a quadratic endo pre-Lie algebra. Then we have
\begin{itemize}
\item[\rm(a)]The adjoint $ \widehat{\phi+ \psi^\ast}$ of $\phi+ \psi^\ast$ with respect to $(\cdot,\cdot)_-$ is $\psi+\phi^\ast$. Furthermore, $\psi+\phi^\ast$ dually represent the endo pre-Lie algebra $(\g\bowtie \g^\ast, \phi+\psi^\ast)$.
\item[\rm(b)] $\psi$ dually represents the endo pre-Lie algebra $(\g, \phi)$.
\item[\rm(c)]$\phi^\ast$ dually represents the endo pre-Lie algebra $(\g^\ast, \psi^\ast)$.
\end{itemize}
\end{lem}
\begin{proof}
For (a), we have
\begin{eqnarray*}
((\phi+ \psi^\ast)(x+a), y+b)_-&=&(\phi(x)+ \psi^\ast(a), y+b)_-\\
&=&\langle \psi^\ast(a), y\rangle-\langle\phi(x), b\rangle\\
&=&\langle a, \psi(y)\rangle -\langle x, \phi^\ast(b)\rangle\\
&=&(a+x, (\psi+\phi^\ast)(y+b))_-,
\end{eqnarray*}
which implies that the adjoint $ \widehat{\phi+ \psi^\ast}$ of $\phi+ \psi^\ast$ with respect to $(\cdot,\cdot)_-$ is $\psi+\phi^\ast$.

For (b), due to (a), we have
\begin{eqnarray*}
( \psi+\phi^\ast)([(\phi+ \psi^\ast)(x+a),(y+b)]_\frkd)&=&[x+a,( \psi+\phi^\ast)(y+b)]_\frkd,\\
( \psi+\phi^\ast)((\phi+ \psi^\ast)(x+a)\ast(y+b))&=&(x+a)\ast( \psi+\phi^\ast)(y+b).
\end{eqnarray*}
Letting  $a=b=0$, have
\begin{eqnarray*}
\psi([\phi(x),y]_\g)=[x,\psi(y)]_\g,\quad \psi(\phi(x)\cdot_\g y)=x\cdot_\g \psi(y),
\end{eqnarray*}
which implies that $\psi$ dually represents the endo pre-Lie algebra $(\g, \phi)$.

For (c), letting $x=y=0$,  we have
\begin{eqnarray*}
\phi^\ast( [\psi^\ast(a),b]_{\g^*})=[a,\phi^\ast(b)]_{\g^*},\quad \phi^\ast( \psi^\ast(a)\cdot_{\g^\ast}b)=a\cdot_{\g^\ast}\phi^\ast(b),
\end{eqnarray*}
which implies that $\phi^\ast$ dually represents the endo pre-Lie algebra $(\g^\ast, \psi^\ast)$.
\end{proof}

\begin{defi}
Let $(\g,\omega_\g)$ be the phase space of the Lie algebra $\g$. Let $\phi:\g\rightarrow\g $ and $\psi:\g\rightarrow\g $ be two  linear maps. If $\phi$ and $\psi$ satisfy
\begin{equation}\label{eq:endo sym2}
(\psi+\phi^*)([(\phi+\psi^*)(x+a), y+b]_p)=[x+a, (\psi+\phi^*)(y+b)]_p,\quad \forall~x,y\in\g,a,b\in\g^*,
\end{equation}
then $(\omega_p,\phi,\psi)$ is called an {\bf endo phase space} of the Lie algebra $\g$.
\end{defi}

\begin{thm}\label{thm:the equivalence of matched pairs and manin triples}
Let $(\g, \phi)$ be an endo pre-Lie algebra. Suppose that there is an endo pre-Lie algebra structure $(\g^\ast,  \psi^\ast)$ on its dual space $\g^\ast$. Then the following statements are equivalent:
\begin{itemize}
  \item[{\rm(i)}] $((\frkd=\g\oplus \g^\ast, \phi+\psi^\ast,(\cdot,\cdot)_-), (\g, \phi), (\g^\ast, \psi^\ast))$ is a Manin triple of endo pre-Lie algebras;
  \item[{\rm(ii)}] $((\g, \phi), (\g^\ast, \psi^\ast), \ad^\ast, -R^\ast, \AD^\ast, -\frkR^\ast))$ is a matched pair of endo pre-Lie algebras;
  \item[{\rm(iii)}]$(\omega_p,\phi,\psi)$ is the endo phase space of the sub-adjacent Lie algebra $\g^c$.
  \end{itemize}
\end{thm}
\begin{proof}
We first show that (i) is equivalent to (ii).

Note that $((\frkd, \phi+\psi^\ast,(\cdot,\cdot)_-), (\g, \phi),(\g^\ast, \psi^\ast))$ is a Manin pair of endo pre-Lie algebras if and only if
\begin{itemize}
\item[\rm(1)] $(\frkd, \phi+\psi^\ast, (\cdot,\cdot)_-)$ is a quadratic endo  pre-Lie algebra,
\item[\rm(2)]  $(\frkd, \g,  \g^\ast)$ is  a Manin triple of pre-Lie algebras associated to $\g$ and $\g^\ast$.
\end{itemize}
Since $(\cdot,\cdot)_-$ is invariant, the pre-Lie algebra structure on $\frkd$ is given by
\begin{eqnarray*}
(x+a)\ast(y+b)=x\cdot_\g y+\ad_x^\ast b-R^*_ya+a\cdot_{\g^\ast}b+\AD^\ast_ay-\frkR^*_b x,\quad x, y\in \g, ~a, b \in \g^\ast.
\end{eqnarray*}
So by Theorem \ref{thm:Manin-Mathched pair}, we only need to prove that $\phi+ \psi^\ast$ is a homomorphism of $(\frkd,\ast)$. By a direct calculation, we have
\begin{eqnarray*}
(\phi+ \psi^\ast)(x+a)\ast(\phi+ \psi^\ast)(y+b)&=&(\phi(x)+ \psi^\ast(a))\ast(\phi(y)+ \psi^\ast(b))\\
&=&\phi(x)\cdot_\g\phi(y)+\ad^\ast_{\phi(x)}( \psi^\ast(b))-R^\ast_{\phi(y)}( \psi^\ast(a))\\
&&+ \psi^\ast(a)\cdot_{\g^\ast} \psi^\ast(b)+\AD_{\psi^\ast(a)}^\ast\phi(y)-\frkR_{\psi^\ast(b)}^\ast\phi(x),\\
(\phi+\psi^\ast)((x+a)\ast(y+b))&=&\phi(x\cdot_\g y)+\phi(\AD_a^\ast y)-\phi(\frkR_b^* x)\\
&&+\psi^\ast(a\cdot_{\g^\ast } b)+\psi^\ast(\ad_x^\ast b)- \psi^\ast(R^*_y a).
\end{eqnarray*}
Therefore, $\phi+ \psi^\ast$ is a homomorphism of $(\frkd, \ast)$ if and only if $(\g^\ast;\ad^\ast, R^\ast, \psi^\ast)$ is a representation of the endo pre-Lie algebra $(\g, \phi)$ and  $(\g; \AD^\ast, \frkR^\ast, \phi)$ is a representation of the endo pre-Lie algebra $(\g^\ast, \psi^\ast)$.

Next, we show that (i) is equivalent to (iii). Since the adjoint $ \widehat{\phi+ \psi^\ast}$ of $\phi+ \psi^\ast$ with respect to $\omega_p$ is $\psi+\phi^\ast$, \eqref{eq:endo sym2} is just \eqref{eq:endo sym1}. Furthermore, by 
Proposition \ref{pro:quadratic endo endo symp}, $(\frkd, \phi+\psi^\ast, \omega_p)$ is a quadratic endo  pre-Lie algebra. The conclusion follows.

\end{proof}

\subsection{Endo pre-Lie bialgebras}
\begin{defi}
A linear space  $\g$ with a linear map $ \alpha:\g\rightarrow\g\otimes\g$ is called a {\bf pre-Lie coalgebra} if for any x in $\g$, we have
\begin{eqnarray}
({\rm id}\otimes\alpha)\alpha(x)-(\tau\otimes {\rm id})({\rm id}\otimes\alpha)\alpha(x)=(\alpha\otimes {\rm id})\alpha(x)-(\tau\otimes {\rm id})(\alpha\otimes {\rm id})\alpha(x),
\end{eqnarray}
where $\tau:\g\otimes\g\rightarrow\g\otimes\g$ is flip map.
\end{defi}

For a Lie algebra $\g$ and a representation $(V;\rho)$ of $\g$, recall that a 1-cocycle associated to $\alpha:\g\rightarrow V$ is a linear map from $\g$ to $V$ satisfying
\begin{eqnarray*}
\alpha([x,y]_\g)=\rho(x)\alpha(y)-\rho(y)\alpha(x),\quad\forall x,y\in\g.
\end{eqnarray*}

\begin{defi}{\rm (\cite{Left-symmetric bialgebras})}
Let $\g$ be a vector space. A {\bf pre-Lie bialgebra} structure on $\g$ is a pair of linear maps $(\alpha, \beta)$ such that $\alpha:\g\rightarrow\g\otimes\g$, $\beta:\g^\ast\rightarrow\g^\ast\otimes\g^\ast$, and for any $ x, y\in \g, ~a, b\in \g^\ast$ , we have
\begin{itemize}
\item[\rm(i)] $\alpha^\ast:\g^\ast\otimes\g^\ast\rightarrow\g^\ast$ is a pre-Lie algebra structure on $\g^\ast$,
\item[\rm(ii)]  $\beta^\ast:\g\otimes\g\rightarrow\g$ is a pre-Lie algebra structure on $\g$,
\item[\rm(iii)]$\alpha$ is a $1$-cocycle of $\g$ associated to $ L\otimes1+1\otimes\ad$ with values in $\g\otimes\g$,
\item[\rm(iv)]$\beta$ is a $1$-cocycle of $\g^\ast$ associated to $\frkL\otimes1+1\otimes\AD$ with values in $\g^\ast\otimes\g^\ast$.
\end{itemize}
We denote pre-Lie bialgebras by $(\g, \g^\ast, \alpha, \beta)$, or simply $(\g, \g^\ast)$.
\end{defi}

\begin{thm}{\rm (\cite{Left-symmetric bialgebras})}\label{thm:Manin-Mathched-bialgebra pair}
	Let $(\g, \cdot_\g)$ be a pre-Lie algebra. Suppose that there is a pre-Lie algebra structure $\cdot_{\g^\ast}$ on its dual space $\g^\ast$ . Then the following statements are equivalent:
	\begin{itemize}
		\item[{\rm(i)}]$(\frkd=\g\oplus \g^\ast, \g, \g^\ast,(\cdot,\cdot)_-)$ is a Manin triple, where $(\cdot,\cdot)_-$ is given by \eqref{symplectic bracket};
		\item[{\rm(ii)}]$(\g, \g^\ast, \ad^\ast, -R^\ast, \AD^\ast, -\frkR^\ast)$ is a matched pair of pre-Lie algebras;
		\item[{\rm(iii)}]$(\g, \g^\ast)$ is a pre-Lie bialgebra.
	\end{itemize}
\end{thm}

\begin{defi}
An {\bf endo pre-Lie coalgebra} is a  pre-Lie coalgebra $(\g, \alpha)$ with a pre-Lie coalgebra endomorphism $\psi:\g\rightarrow\g$, that is,
\begin{eqnarray}\label{eq:pre-Lie coalgebra homomorphism}
( \psi\otimes \psi)\alpha=\alpha \psi.
\end{eqnarray}
Under finite-dimension condition, \eqref{eq:pre-Lie coalgebra homomorphism} is equivalent to the condition that  $ \psi^\ast:\g^\ast\rightarrow\g^\ast$ is an endomorphism of the pre-Lie algebra  $\g^\ast$.
\end{defi}

\begin{defi}\label{defi:endo pre-Lie bialgebra}
An {\bf endo pre-Lie bialgebra} is $((\g, \cdot_\g),(\g^\ast,\cdot_{\g^\ast}), \alpha, \beta, \phi,  \psi)$, or simply $((\g,\phi),(\g^\ast,\psi),\alpha,\beta)$ in which
\begin{itemize}
\item[\rm(i)] $(\g, \g^\ast, \alpha, \beta)$ is a pre-Lie bialgebra,
\item[\rm(ii)]  $(\g, \phi)$ is  an endo pre-Lie algebra,
\item[\rm(iii)] $(\g, \alpha,  \psi)$ is an endo pre-Lie coalgebra,
\item[\rm(iv)] the the following compatibility conditions holds:
\begin{eqnarray}
	({\rm id}\otimes\phi)\circ \alpha&=&( \psi\otimes {\rm id})\circ \alpha\circ \phi,\label{eq:compatibility conditions 1}\\
	(\phi \otimes {\rm id})\circ \alpha&=&( {\rm id}  \otimes \psi)\circ \alpha\circ \phi,\label{eq:compatibility conditions 12}\\
	\psi(x\cdot_\g \phi(y))&=&\psi(x) \cdot_\g y ,\label{eq:compatibility conditions 2}\\
	\psi(\phi(x)\cdot_\g  y)&=&x\cdot_\g \psi(y),\label{eq:compatibility conditions 3}~\forall x, y\in \g.
\end{eqnarray}
\end{itemize}
\end{defi}

\begin{thm}\label{thm:equvialent}
	Let $(\g,\cdot_\g, \phi)$ be an endo pre-Lie algebra. Suppose that there is an  endo pre-Lie algebra $(\g^\ast, \cdot_{\g^\ast} ,\psi^\ast)$ on the linear dual $\g^\ast$ of  $\g$.  Then the following statements are equivalent:
	\begin{itemize}
		\item[\rm(i)]$((\frkd=\g\oplus \g^\ast, \phi+\psi^\ast,(\cdot,\cdot)_-), (\g, \phi), (\g^\ast, \psi^\ast))$ is a Manin triple of endo pre-Lie algebras;
		\item[\rm(ii)]$((\g, \phi), (\g^\ast, \psi^\ast), \ad^\ast, -R^\ast, \AD^\ast, -\frkR^\ast)$ is a matched pair of endo pre-Lie algebras;
 \item[{\rm(iii)}]$(\omega_p,\phi,\psi)$ is the endo phase space of the sub-adjacent Lie algebra $\g^c$;
		\item[\rm(iv)] $(\g, \g^\ast)$ is a endo pre-Lie bialgebra.
	\end{itemize}
\end{thm}
\begin{proof}
	By Theorem \ref{thm:the equivalence of matched pairs and manin triples},   (i), (ii) and (iii) are equivalent.
	
	By Theorem \ref{thm:Manin-Mathched-bialgebra pair},  $(\g, \g^\ast, \ad^\ast, -R^\ast, \AD^\ast, -\frkR^\ast)$ is a matched pair of pre-Lie algebras if and only if $(\g, \g^\ast)$ is a pre-Lie bialgebra. Furthermore, $\phi+\psi^\ast$ is a pre-Lie algebra homomorphism if and only if \eqref{eq:compatibility conditions 1}-\eqref{eq:compatibility conditions 3} hold. Thus (ii) is equivalent to (iv).
\end{proof}

Let $\g$ be a pre-Lie algebra. For a given $r\in \g\otimes\g$, define $\alpha_r:\g\rightarrow\g\otimes\g$ by
\begin{eqnarray}\label{eq:map}
 \alpha_r(x)=(L_x\otimes {\rm id}+{\rm id}\otimes \ad_x)r ,~\forall x\in\g.
 \end{eqnarray}

The coboundary  pre-Lie bialgebra is an important class of pre-Lie bialgebras.

 \begin{pro}{\rm (\cite{Left-symmetric bialgebras})}\label{pro:coboundary  pre-Lie bialgebra}
 Let $(\g, \cdot_\g)$ be a pre-Lie algebra and $r\in \g\otimes\g$. Then  $(\g, \g^\ast, \alpha, \beta)$ is a  pre-Lie bialgebra, which is called a {\bf coboundary  pre-Lie bialgebra} if and only if  for any $x\in\g$, we have
 \begin{eqnarray}
 (P(x, y)-P(x)P(y))(r_{12}-r_{21})&=&0,\label{eq:coboundary condition 1}\\
Q(x)[[r, r]]&=&0\label{eq:coboundary condition 2},
 \end{eqnarray}
where \begin{eqnarray*}
    P(x)&=&L_x\otimes {\rm id}+{\rm id}\otimes L_x,\\
    Q(x)&=&L_x\otimes {\rm id}\otimes {\rm id}+{\rm id}\otimes L_x\otimes {\rm id}+{\rm id}\otimes {\rm id}\otimes L_x,\\~
   [[r, r]]&=&r_{13}\cdot r_{12}-r_{23}\cdot r_{21}+[r_{23}, r_{12}]-[r_{13}, r_{21}]-[r_{13}, r_{23}].
   \end{eqnarray*}
 \end{pro}

\begin{defi}
An endo pre-Lie bialgebra  $((\g,\cdot_\g), (\g^\ast,\cdot_{\g^*}), \alpha, \beta, \phi,  \psi)$ is called {\bf coboundary} if $ \alpha=\alpha_r $, for any $\ r\in\g\otimes\g$.
\end{defi}

Then we have
\begin{thm}\label{thm:equivalent theorem}
Let $(\g,\cdot_\g, \phi)$ be an endo  pre-Lie algebra, $ \psi$ dually represent $(\g,\cdot_\g, \phi)$, and $r\in\g\otimes\g$. Then linear map $ \alpha_r$ induces an endo pre-Lie bialgebra   $((\g,\cdot_\g), (\g^\ast,\cdot_{\g^*}), \alpha, \beta, \phi,  \psi)$ if and only if $r$ satisfies  \eqref{eq:coboundary condition 1} and  \eqref{eq:coboundary condition 2},  and for any $x\in\g$, the following conditions hold:
\begin{eqnarray}
(\psi L_x\otimes {\rm id})({\rm id}\otimes \psi-\phi\otimes {\rm id})(r)+({\rm id}\otimes \psi\ad_x)( \psi\otimes {\rm id}-{\rm id}\otimes\phi)(r)&=&0,\label{eq:conditions 1}\\
 (L_x\otimes {\rm id}+{\rm id}\otimes\ad_{\phi(x)})({\rm id}\otimes\phi- \psi\otimes {\rm id})(r)&=&0\label{eq:conditions 2},\\
 (L_x\otimes {\rm id}+{\rm id}\otimes\ad_{\phi(x)})(\phi\otimes{\rm id}-{\rm id}\otimes\psi)(r)&=&0\label{eq:conditions 3}.
 \end{eqnarray}
\end{thm}

\begin{proof}
By Definition \ref{defi:endo pre-Lie bialgebra} and Proposition \ref{pro:coboundary  pre-Lie bialgebra}, $((\g,\cdot_\g), (\g^\ast,\cdot_{\g^*}), \alpha, \beta, \phi,  \psi)$ is an endo pre-Lie bialgebra if and only if $r$ satisfies  \eqref{eq:coboundary condition 1} and  \eqref{eq:coboundary condition 2}, and  \eqref{eq:pre-Lie coalgebra homomorphism} and \eqref{eq:compatibility conditions 1} hold. Let $r=\Sigma a_i\otimes b_i $, for any $x\in\g$,  we have
\begin{eqnarray*}
&&( \psi\otimes \psi)\alpha_r(x)-\alpha_r( \psi(x))\\
&=&( \psi\otimes \psi)(L_x\otimes {\rm id}+{\rm id}\otimes\ad_x)\Sigma ( a_i\otimes b_i)-(L_{ \psi(x)}\otimes {\rm id} +{\rm id}\otimes\ad_{ \psi(x)})\Sigma (a_i\otimes b_i)\\
&=&( \psi\otimes \psi)\Sigma (x\cdot a_i\otimes b_i+a_i\otimes [x, b_i])-\Sigma( \psi(x)\cdot a_i\otimes b_i+a_i\otimes [ \psi(x), b_i])\\
&=&\Sigma( \psi(x\cdot a_i)\otimes \psi(b_i)+ \psi(a_i)\otimes \psi([x, b_i])- \psi(x)\cdot a_i\otimes b_i-a_i\otimes[ \psi(x), b_i])\\
&=&( \psi L_x\otimes {\rm id})({\rm id} \otimes \psi)(r)+({\rm id}\otimes \psi\ad_x)( \psi\otimes {\rm id} )(r)-\Sigma( \psi (x\cdot \phi(a_i))\otimes b_i+a_i\otimes \psi[x, \phi(b_i)])\\
&=&( \psi L_x\otimes {\rm id})({\rm id} \otimes \psi)(r)+({\rm id}\otimes \psi\ad_x)( \psi\otimes {\rm id})(r)-( \psi L_x\otimes {\rm id})(\phi\otimes {\rm id})(r)-({\rm id}\otimes \psi\ad_x)({\rm id}\otimes\phi)(r)\\
&=&( \psi L_x\otimes {\rm id})({\rm id} \otimes \psi-\phi\otimes {\rm id})(r)+({\rm id}\otimes \psi\ad_x)( \psi\otimes {\rm id}-{\rm id}\otimes\phi)(r),
 \end{eqnarray*}
which implies that \eqref{eq:pre-Lie coalgebra homomorphism} holds if and only if  \eqref{eq:conditions 1} holds.
Similarly, we have
\begin{eqnarray*}
({\rm id}\otimes\phi)\alpha_r(x)-( \psi\otimes {\rm id} )\alpha_r(\phi(x))
&=&(L_x\otimes {\rm id}+{\rm id}\otimes\ad_{\phi(x)})({\rm id}\otimes\phi- \psi\otimes {\rm id})(r),\\
(\phi\otimes{\rm id})\alpha_r(x)-({\rm id}\otimes\psi )\alpha_r(\phi(x))
&=&(L_x\otimes {\rm id}+{\rm id}\otimes\ad_{\phi(x)})(\phi\otimes{\rm id}-{\rm id}\otimes\psi)(r).
 \end{eqnarray*}
which implies that \eqref{eq:compatibility conditions 1} holds if and only if  \eqref{eq:conditions 2} holds, and  \eqref{eq:compatibility conditions 12} holds if and only if  \eqref{eq:conditions 3} holds
\end{proof}

\emptycomment{For coherent homomorphism of pre-Lie bialgebras, by Proposition \ref{pro:coherent endo and endo pre-Lie bia}, we have
\begin{cor}\label{cor:condition of equivalence}
Let $(\g, \cdot_\g)$  be a pre-Lie algebra  and $\phi:\g\rightarrow\g$ be a  pre-Lie algebra homomorphism. Let $ \psi:\g\rightarrow\g$ be a linear map satisfying  \eqref{eq:compatibility conditions 2} and \eqref{eq:compatibility conditions 3} and $r\in\g\otimes\g$. Then $(\g, \g^\ast, \alpha_r, \beta)$ is a  pre-Lie bialgebra and $(\phi, \psi)$ is a coherent endomorphism of pre-Lie bialgebras if and only if  \eqref{eq:coboundary condition 1}-\eqref{eq:conditions 3} are satisfied.
\end{cor}}

By Theorem \ref{thm:equivalent theorem}, we have
\begin{cor}\label{cor:endo YB equation}
Let  $(\g,\cdot_\g, \phi)$   be an endo pre-Lie algebra and $ \psi$ dually represent $(\g,\cdot_\g, \phi)$.  Let $r\in\g\otimes\g$, then the linear map $\alpha_r$ given by \eqref{eq:map} induces an endo pre-Lie bialgebra $((\g,\cdot_\g), (\g^\ast,\cdot_{\g^*}), \alpha, \beta, \phi,  \psi)$  if $r$ satisfies \eqref{eq:coboundary condition 1} and the following equations hold:
\begin{eqnarray}
[[r, r]]=r_{13}\cdot r_{12}-r_{23}\cdot r_{21}+[r_{23}, r_{12}]-[r_{13}, r_{21}]-[r_{13}, r_{23}]&=&0,\label{eq:equation 1}\\
(\phi\otimes {\rm id}-{\rm id}\otimes \psi)(r)&=&0,\label{eq:equation 2}\\
( \psi\otimes {\rm id}-{\rm id}\otimes\phi)(r)&=&0.\label{eq:equation 3}
 \end{eqnarray}
 \end{cor}

Equation \eqref{eq:equation 1} is the classical Yang-Baxter equation (CYBE) in a pre-Lie algebra and a solution of the CYBE in a pre-Lie algebra is called a classical $\frks$-matrices. A pre-Lie bialgebra $(\g, \g^\ast, \alpha, \beta)$ is called {\bf quasi-triangular} if it is obtained from a classical $\frks$-matrices $r$ given by \eqref{eq:map} and is called {\bf triangular} if it is obtained from symmetric classical $\frks$-matrices r (i.e. $r=\tau(r))$. Moreover, if $r$ is symmetric, then \eqref{eq:equation 2} holds if and only if \eqref{eq:equation 3}  holds.

\begin{defi}\label{defi:classical Yang-Baxter equation}
Let $(\g,\cdot_\g, \phi)$ be an endo  pre-Lie algebra, $ \psi$ dually represent $(\g,\cdot_\g, \phi)$, and $r\in\g\otimes\g$. Then \eqref{eq:equation 1} with conditions given by \eqref{eq:equation 2} and \eqref{eq:equation 3} are  called the {\bf $ \psi$-classical Yang-Baxter equation} {\bf ($ \psi$-CYBE)} in   $(\g,\cdot_\g, \phi)$.
\end{defi}

Let $(\g,\cdot_\g, \phi)$ be an endo  pre-Lie algebra, $ \psi$ dually represent $(\g,\cdot_\g, \phi)$, and $r\in\g\otimes\g$. If $r$ is a solution of the $ \psi$-classical Yang-Baxter equation and satisfies \eqref{eq:coboundary condition 1},  by Corollary \ref{cor:endo YB equation}, the linear map $\alpha_r$ given by \eqref{eq:map} induces an endo pre-Lie bialgebra $((\g,\cdot_\g), (\g^\ast,\cdot_{\g^*}), \alpha, \beta, \phi,  \psi)$.

\section{Coherent homomorphisms of pre-Lie bialgebras, Manin triples and phase spaces of Lie algebras}\label{sec:coherent endoprelie}
As the main motivation and application of our study
of endo pre-Lie bialgebras, Manin triples, matched pairs and phase space spaces of Lie algebras, we introduce new notions of homomorphisms of these structures ensuring compatibility with structural correspondences and enabling the organization of these objects into categories so that the maps among the classes become functors or category equivalences.

\subsection{Coherent homomorphisms of phase spaces and Manin triples for pre-Lie algebras }

Let $\g$ be a Lie algebra and $\omega$ be a skew-symmetric non-degenerate bilinear form. Define $\omega^\natural:\g\rightarrow\g^*$ by
$$\langle\omega^\natural(x),y\rangle=\omega(x,y),\quad\forall~x,y\in \g.$$
Furthermore, we let $T_\omega=(\omega^\natural)^{-1}$. Then $\omega$ is a $2$-cocycle if and only if $T_\omega:\g^*\rightarrow \g$ is a relative Rota-Baxter operator with respective to the coadjoint representation $(\g^*;\ad^*)$.

In order to build the relationships between symplectic Lie algebras and relative Rota-Baxter operators on the level of categories, we introduce the notion of coherent homomorphisms between symplectic Lie algebras. 

\begin{defi}
Let $(\g,\omega_\g)$ and $(\h,\omega_\h)$ be two symplectic Lie algebras. A {\bf coherent homomorphism} from $(\g,\omega_\g)$ to $(\h,\omega_\h)$ consists of a  Lie algebra homomorphism $\phi:\g\rightarrow\h$ and a linear map $ \psi:\h\rightarrow\g$ such that the following equalities holds:
\begin{eqnarray}
 \label{eq:weak symp1} [x,\psi(y)]_\g &=& \psi[\phi(x),y]_\h, \\
 \label{eq:weak symp2} \omega_\h(\phi(x),y) &=& \omega_\g(x,\psi(y)), \quad\forall~x\in\g,y\in \h.
\end{eqnarray}
\end{defi}

\begin{pro} \label{pro:weak-homomorphsim symp}
Let $(\g,\omega_\g)$ and $(\h,\omega_\h)$ be two symplectic Lie algebras. Let $\phi:\g\rightarrow \h$ and $\psi:\h\rightarrow \g$ be linear maps. Then  
$(\phi, \psi)$ is a coherent homomorphism of symplectic Lie algebras from $(\g,\omega_\g)$ to $(\h,\omega_\h)$ if and only if $(\phi,\psi^*)$ is a homomorphism between
the corresponding relative Rota-Baxter operators $T_{\omega_\g}$ and $T_{\omega_\h}$. 

This correspondence defines an equivalence from the category ${\bf
SLA}$ of symplectic Lie algebras to the category ${\bf IRB}$ of
invertible relative Rota-Baxter operators on the Lie algebras associated to the representations $\ad^*$ satisfying $\langle T_{\omega}(x),y\rangle=-\langle T_{\omega}^\sharp(y),x\rangle$ for $x,y\in \g$ or $\h$.
\end{pro}
\begin{proof}
It follows by a direct calculation.
\end{proof}

The benefit of coherent homomorphisms of Manin triples derived from Manin triples of endo pre-Lie algebras is that it is compatible with the naturally defined morphisms of pre-Lie bialgebras.
\begin{defi}\label{defi:coherent homomorphisms of Manin triples}
Let $((\g\bowtie\g^\ast,(\cdot,\cdot)_ -), \g, \g^\ast)$ and $((\h\bowtie\h^\ast,(\cdot,\cdot)_-), \h, \h^\ast)$ be Manin triples of pre-Lie algebras. A {\bf coherent  homomorphism} between them is a pre-Lie algebra homomorphism
\begin{eqnarray*}
f:\g\bowtie\g^\ast\rightarrow\h\bowtie\h^\ast,
\end{eqnarray*}
 that restricts to pre-Lie algebra homomorphisms
 \begin{eqnarray*}
f\mid_\g:\g\rightarrow\h, \quad f\mid_{\g^\ast}:\g^\ast\rightarrow\h^\ast.
 \end{eqnarray*}
 If $f$ is bijective, it is called a {\bf coherent isomorphism of  Manin triples}. Let  {\bf MT} denote the category of Manin triples with coherent  homomorphisms.
\end{defi}

By the definition of Manin triple of endo pre-Lie algebras, we have
\begin{pro}
Let $((\g\bowtie\g^\ast,(-,-)_ -), \g, \g^\ast)$ be a Manin triple of pre-Lie algebras. Then a linear map $f\in \End(\g\bowtie\g^\ast)$ is a coherent endomorphism of the Manin triple if and only if $((\g\bowtie\g^\ast, f), (\g, f\mid_\g), (\g^\ast, f\mid_{\g^\ast}))$ is a Manin triple of endo pre-Lie algebras associated to $(\g, f\mid_\g)$ and $(\g^\ast, f\mid_{\g^\ast})$.
\end{pro}

\begin{defi}\label{defi:coherent homomorphisms of phase spaces}
Let $(\g,\omega_\g)$ and $(\h,\omega_\h)$ be phase spaces of Lie algebras $\g$ and $\h$, respectively. A {\bf coherent  homomorphism} between them is a linear map
\begin{eqnarray*}
f:\g\bowtie\g^\ast\rightarrow\h\bowtie\h^\ast,
\end{eqnarray*}
 that restricts to linear maps
 \begin{eqnarray*}
f\mid_\g:\g\rightarrow\h, \quad f\mid_{\g^\ast}:\g^\ast\rightarrow\h^\ast,
 \end{eqnarray*}
 satisfying
 \begin{equation}\label{eq:endo sym2}
f^*([f^*(x+a), y+b]_p)=[x+a, f^*(y+b)]_p,\quad \forall~x\in\g,a\in\g^*,y\in\h,b\in\h^*.
\end{equation}
 If $f$ is bijective, it is called a {\bf coherent isomorphism of  phase spaces}. Let  {\bf PS} denote the category of phase spaces with coherent  homomorphisms.
\end{defi}

\begin{pro}
Let $(\g,\omega_\g)$ and $(\h,\omega_\h)$ be phase spaces of Lie algebras $\g$ and $\h$, respectively. Then \begin{eqnarray*}
f:\g\bowtie\g^\ast\rightarrow\h\bowtie\h^\ast,
\end{eqnarray*} 
is a coherent  homomorphism between them if and only if $(f,f^*)$ is a coherent homomorphism of symplectic Lie algebras from $(\g,\omega_p)$ to $(\h,\omega_p)$.
\end{pro}
\begin{proof}
By a direct calculation, we have
\begin{eqnarray*}
\omega_\h(f(x+a),y+b)&=&\omega_\g(x+a,f^*(y+b)),\quad \forall~x\in\g,a\in\g^*,y\in\h,b\in\h^*.
\end{eqnarray*}
Furthermore, \eqref{eq:endo sym2} is just \eqref{eq:weak symp1}. The conclusion follows.
\end{proof}

\begin{pro}
Let $(\g,\omega_p)$ be a phase space of a Lie algebra $\g$. Then a linear map $f\in \End(\g\bowtie\g^\ast)$ is a coherent endomorphism of the phase space if and only if $(\g\bowtie\g^\ast,\omega_p, f\mid_\g, f^*\mid_{\g})$ is an endo phase space of the Lie algebra $\g$.
\end{pro}
\begin{proof}
It follows by a direct calculation.
\end{proof}

Due to the correspondence between endo phase spaces of Lie algebras and Manin triples of endo pre-Lie algebras given in Theorem \ref{thm:the equivalence of matched pairs and manin triples}, we have
\begin{pro}
Let $(\g,\omega_p)$ be a phase space of a Lie algebra $\g$ and $(\frkd=\g\oplus\g^*,(\cdot,\cdot)_-=\omega_p)$ its corresponding quadratic pre-Lie algebra. Then $f$ is a coherent endomorphism of the phase space $(\g, \omega_p)$ if and only if $f$ is a coherent endomorphism of the  Manin triple $((\frkd,(-,-)_ -), \g, \g^\ast)$.
\end{pro}

Furthermore, we have
\begin{pro}\label{pro:category phase-Manin}
\begin{itemize}
\item[\rm(i)] Let $(\g, \omega_\g)$ and $(\h, \omega_\h)$ be phase spaces and, let $((\g\bowtie\g^\ast,(-,-)_ -), \g, \g^\ast)$ and $((\h\bowtie\h^\ast,(-,-)_ -), \h, \h^\ast)$ be the corresponding Manin triples of pre-Lie algebras.  There is a bijection between the set $\Hom_{\bf PS}(\g, \h)$ of coherent homomorphisms between the phase spaces and the set $\Hom_{\bf MT}(\g\bowtie\g^\ast, \h\bowtie\h^\ast)$ of coherent homomorphisms between the Manin triples, where  the bijection is give by sending $f$ to $f$.
\item[\rm(ii)] The correspondence in {\rm (i)} gives an equivalence from the category {\bf PS} of phase spaces to the category {\bf MT} of Manin triples.
\end{itemize}
\end{pro}

\subsection{Coherent homomorphisms of pre-Lie bialgebras, Manin triples and matched pairs }
As both main motivation and application of our study of endo pre-Lie bialgebras, we introduce a notion of homomorphism of pre-Lie bialgebras,  that is compatible with those of Manin triples and matched pairs.
\begin{defi}
A {\bf coherent endomorphism} on a pre-Lie bialgebras $(\g, \g^\ast, \alpha, \beta)$ consists of a pre-Lie algebra homomorphism $\phi:\g\rightarrow\g$ and a pre-Lie coalgebra homomorphism $ \psi:\g\rightarrow\g$ satisfying \eqref{eq:compatibility conditions 1}-\eqref{eq:compatibility conditions 3}.
\end{defi}
Then we have
\begin{pro}\label{pro:coherent endo and endo pre-Lie bia}
 $((\g, \cdot_\g),(\g^\ast,\cdot_{\g^\ast}), \alpha, \beta, \phi,  \psi)$ is an endo pre-Lie bialgebra if and only if $(\phi,  \psi)$ is a coherent endomorphism of the pre-Lie bialgebras $(\g, \g^\ast, \alpha, \beta)$.
\end{pro}
\begin{defi}\label{defi:A coherent endomorphism of pre-Lie bialgebras}
Let $(\g, \g^\ast, \alpha_\g, \beta_\g)$ and $(\h, \h^\ast, \alpha_\h, \beta_\h)$ be pre-Lie bialgebras. {\bf A coherent homomorphism of pre-Lie bialgebras} from $(\g, \g^\ast, \alpha_\g, \beta_\g)$ to $(\h, \h^\ast, \alpha_\h, \beta_\h)$ is a pair  $(\phi, \psi)$ of linear maps such that
\begin{itemize}
\item[\rm(i)] $\phi:\g\rightarrow\h$ is a homomorphism of pre-Lie algebras,
\item[\rm(ii)]  $\psi:\h\rightarrow\g$ is  a homomorphism of pre-Lie coalgebras,
\item[\rm(iii)] the following equations hold:
\begin{eqnarray}
	({\rm id}_\g\otimes\phi)\circ \alpha_\g&=&( \psi\otimes {\rm id}_\h)\circ \alpha_\h\circ \phi,\label{eq:the polarization 1}\\
	(\phi \otimes {\rm id}_\g)\circ \alpha_\g&=&( {\rm id}_\h  \otimes \psi)\circ \alpha_\h\circ \phi,\label{eq:the polarization 12}\\
	\psi(y\cdot_\h \phi(x))&=&\psi(y) \cdot_\g x,\label{eq:the polarization 2}\\
	\psi(\phi(x)\cdot_\h  y)&=&x\cdot_\g \psi(y),\label{eq:the polarization 3}~\forall x\in\g, y\in \h.
\end{eqnarray}
\end{itemize}
If both $\phi$ and $ \psi$ are bijective, the pair is called {\bf coherent isomorphism of  pre-Lie bialgebras}. Let {\bf PLB} denote the category of pre-Lie algebras with coherent homomorphisms.
\end{defi}

Due to the correspondence between endo pre-Lie bialgebras and Manin triples of endo pre-Lie algebras given in Theorem \ref{thm:equvialent}, we have
\begin{pro}
Let $(\g, \g^\ast, \alpha, \beta)$ be a  pre-Lie bialgebra and $((\g\bowtie\g^\ast,(-,-)_ -), \g, \g^\ast)$ be the corresponding Manin triple. Then $(\phi,  \psi)$ is a coherent endomorphism of the pre-Lie bialgebra $(\g, \g^\ast, \alpha, \beta)$ if and only if $\phi+ \psi^\ast$ is a coherent endomorphism of the  Manin triple $((\g\bowtie\g^\ast,(-,-)_ -), \g, \g^\ast)$.
\end{pro}

\begin{pro}\label{pro:all the spaces are finite dimensional}
\begin{itemize}
\item[\rm(i)] Let $(\g, \g^\ast, \alpha_\g, \beta_\g)$ and $(\h, \h^\ast, \alpha_\h, \beta_\h)$ be pre-Lie bialgebras and, let $((\g\bowtie\g^\ast,(-,-)_ -), \g, \g^\ast)$ and $((\h\bowtie\h^\ast,(-,-)_ -), \h, \h^\ast)$ be the corresponding Manin triples of pre-Lie algebras.  There is a bijection between the set $\Hom_{\bf PLB}(\g, \h)$ of coherent homomorphisms between the pre-Lie bialgebras and the set $\Hom_{\bf MT}(\g\bowtie\g^\ast, \h\bowtie\h^\ast)$ of coherent homomorphisms between the Manin triples, where  the bijection is give by sending $(\phi, \psi)$ to $\phi+ \psi^\ast$.
\item[\rm(ii)] The correspondence in {\rm (i)} gives an equivalence from the category {\bf LB} of pre-Lie bialgebras to the category {\bf MT} of Manin triples.
\end{itemize}
\end{pro}
\begin{proof}
{\rm(i)}  Asumme that $(\phi,  \psi)$ is a coherent homomorphism of the pre-Lie bialgebras. For any $ x, y\in \g, a, b\in\g^\ast $ and for $\phi+ \psi^\ast:\g\bowtie\g^\ast\rightarrow \h\bowtie\h^\ast$, we have
\begin{eqnarray*}
&&(\phi+ \psi^\ast)((x+a)\ast(y+b))\\
&=&(\phi+ \psi^\ast)(x\cdot_\g y+\AD^\ast(a)y-\frkR^\ast(b)x+a\cdot_{\g^\ast}b+\ad^\ast(x)b-R^\ast(y)a)\\&=&\phi(x\cdot_\g y+\AD^\ast(a)y-\frkR^\ast(b)x)+ \psi^\ast(a\cdot_{\g^\ast}b+\ad^\ast(x)b-R^\ast(y)a),\\
&&(\phi+ \psi^\ast)(x+a)\ast(\phi+ \psi^\ast)(y+b)\\
&=&(\phi(x)+ \psi^\ast(a))\ast(\phi(y)+ \psi^\ast(b))\\
&=&\phi(x)\cdot_\g\phi(y)+\AD^\ast( \psi^\ast(a))\phi(y)
-\frkR^\ast( \psi^\ast(b))\phi(x)\\
&&+\psi^\ast(a)\cdot_{\g^\ast} \psi^\ast(b)+\ad^\ast(\phi(x))( \psi^\ast(b))-R^\ast(\phi(y))( \psi^\ast(a)).
\end{eqnarray*}
Since $\phi:\g\rightarrow\h$ and $ \psi^\ast:\g^\ast\rightarrow\h^\ast$ are homomorphisms of pre-Lie algebras, we have
\begin{eqnarray*}
\phi(x\cdot_\g y)=\phi(x)\cdot_\g\phi(y) ,\quad  \psi^\ast(a\cdot_{\g^\ast}b)= \psi^\ast(a)\cdot_{\g^\ast} \psi^\ast(b).
\end{eqnarray*}
By \eqref{eq:the polarization 1}, we have
\begin{eqnarray*}
\phi(\AD^\ast(a)y)=\AD^\ast( \psi^\ast(a))\phi(y) , \quad \phi(-\frkR^\ast(b)x)=-\frkR^\ast( \psi^\ast(b))\phi(x).
\end{eqnarray*}
By \eqref{eq:the polarization 2}, we have
\begin{eqnarray*}
 \psi^\ast(\ad^\ast(x)b)=\ad^\ast(\phi(x)) \psi^\ast(b), \quad \psi^\ast(-R^\ast(y)a)=-R^\ast(\phi(y)) \psi^\ast(a).
\end{eqnarray*}
Combining with the above equations, we have
$$(\phi+ \psi^\ast)((x+a)\ast(y+b))=(\phi+ \psi^\ast)(x+a)\ast(\phi+ \psi^\ast)(y+b),$$
which implies that $\phi+ \psi^\ast$ is a coherent homomorphism of the Manin triples.

Conversely, by a similar argument,  if $f$ is a coherent homomorphism of Manin triples, then $(f\mid_\g, (f\mid_{\g^\ast})^\ast)$ is a coherent homomorphism of the corresponding pre-Lie bialgebras.

{\rm(ii)} follows from {\rm(i)} directly.
\end{proof}

Due to the correspondence between matched pairs and Manin triples of  pre-Lie algebras, by Definition \ref{defi:coherent homomorphisms of Manin triples} we can define the coherent homomorphism of matched pairs of pre-Lie algebras directly.
\begin{defi}\label{defi:coherent homomorphism}
Let $(\g, \g^\ast, \ad^\ast, -R^\ast, \AD^\ast, -\frkR^\ast)$ and $(\h, \h^\ast, \ad_\h^\ast, -R_\h^\ast, \AD_{\h^\ast}^\ast, -\frkR_{\h^\ast}^\ast)$ be matched pairs of pre-Lie algebras. A {\bf coherent homomorphism} between them is a  pre-Lie algebra homomorphism
\begin{eqnarray*}
f:\g\bowtie\g^\ast\rightarrow\h\bowtie\h^\ast,
\end{eqnarray*}
that restricts to pre-Lie algebra homomorphisms
 \begin{eqnarray*}
f\mid_\g:\g\rightarrow\h, f\mid_{\g^\ast}:\g^\ast\rightarrow\h^\ast.
 \end{eqnarray*}
 If $f$ is bijection,  it is called a {\bf coherent isomorphism of matched pairs}. Let {\bf MP} denote the category of such matched pairs of pre-Lie algebras with the coherent homomorphisms.
\end{defi}

On the other hand,  by Theorem \ref{thm:endo pre-Lie algebra} we also have the following compatibility of coherent homomorphisms of matched pairs.
\begin{pro}
Let $(\g, \g^\ast, \ad^\ast, -R^\ast, \AD^\ast, -\frkR^\ast)$ be a matched pair of pre-Lie algebras. Then a linear map $f:\g\bowtie\g^\ast\rightarrow\g\bowtie\g^\ast$ is a coherent endomorphism of this matched pair of pre-Lie algebras if and only if $((\g, f\mid_\g), (\g^\ast, f\mid_{\g^\ast}), \ad^\ast, -R^\ast, \AD^\ast, -\frkR^\ast)$ is a matched pair of endo pre-Lie algebra.
\end{pro}

Furthermore, we have
\begin{pro}\label{eq:MP-MT}
\begin{itemize}
\item[\rm(i)] Let $((\g\bowtie\g^\ast,(\cdot,\cdot)_-) ,\g, \g^\ast)$ and $((\h\bowtie\h^\ast,(\cdot,\cdot)_-) ,\h, \h^\ast)$ be Manin triples of pre-Lie algebras, and let  $(\g, \g^\ast, \ad^\ast, -R^\ast, \AD^\ast, -\frkR^\ast)$ and $(\h, \h^\ast, \ad_\h^\ast, -R_\h^\ast, \AD_{\h^\ast}^\ast, -\frkR_{\h^\ast}^\ast)$ be the corresponding  matched pairs of pre-Lie algebras. Then a linear map $f:\g\bowtie\g^\ast\rightarrow\h\bowtie\h^\ast$ is a coherent homomorphism of Manin triples if and only if $f$ is a coherent homomorphisms of matched pairs.
\item[\rm(ii)] The correspondence in {\rm (i)} gives an equivalence from the category {\bf MP} of matched pairs of  the form $(\g, \g^\ast, \ad^\ast, -R^\ast, \AD^\ast, -\frkR^\ast)$ to the category {\bf MT} of Manin triples.
\end{itemize}
\end{pro}
\begin{proof}
By Definition \ref{defi:coherent homomorphisms of Manin triples} and \ref{defi:coherent homomorphism}, (i) holds.

 By (i), there is a bijection between the set of coherent homomorphisms $f:\g\bowtie\g^\ast\rightarrow\h\bowtie\h^\ast$  of Manin triples and the set of coherent homomorphisms $f:\g\bowtie\g^\ast\rightarrow\h\bowtie\h^\ast$ of matched pairs by sending $f$ to $f$ itself, which imply (ii) holds.
\end{proof}

By Proposition \ref{pro:category phase-Manin}, Proposition \ref{pro:all the spaces are finite dimensional} and Proposition \ref{eq:MP-MT},  the following categories are equivalent.
\begin{thm}\label{thm:main theorem}
\begin{itemize}
\item[\rm(i)] the category {\bf PLB} of pre-Lie bialgebras;
\item[\rm(ii)] the category {\bf MT} of Manin triples;
\item[\rm(iii)] the category {\bf PS} of phase spaces;
\item[\rm(iv)]the category {\bf MP} of matched pairs of the form  $(\g, \g^\ast, \ad^\ast, -R^\ast, \AD^\ast, -\frkR^\ast)$.
\end{itemize}
\end{thm}

The following definition of homomorphism of pre-Lie bialgebras are similar to the definition of homomorphism of Lie bialgebras.
\begin{defi}\label{defi:homomorphism of pre-Lie bialgebras}
{\bf A homomorphism of pre-Lie bialgebras} from $(\g, \g^\ast, \alpha_\g, \beta_\g)$ to $(\h, \h^\ast, \alpha_\h, \beta_\h)$ is a linear map $ f:\g\rightarrow\h$ such that  $f$ is both a pre-Lie algebra homomorphism and a pre-Lie coalgebra homomorphism.  Furthermore, if $f$ is also bijective, then $f$ is called an {\bf isomorphism of pre-Lie bialgebras}.
\end{defi}

In the bijective case, we have
\begin{pro}
Let $(\g, \g^\ast, \alpha_\g, \beta_\g)$ and $(\h, \h^\ast, \alpha_\h, \beta_\h)$ be two pre-Lie bialgebras. Then $(\g, \g^\ast, \alpha_\g, \beta_\g)$ is isomorphic to  $(\h, \h^\ast, \alpha_\h, \beta_\h)$ if and only if there exists a pre-Lie algebra isomorphism $\phi:\g\rightarrow\h$ such that $(\phi, \phi^{-1})$ is a coherent isomorphism.
\end{pro}
\begin{proof}
Take $ \psi=\phi^{-1}$, then \eqref{eq:the polarization 1}-\eqref{eq:the polarization 3} hold automatically. Moreover, $\phi^\ast$ is an isomorphism of pre-Lie algebras if and only if $(\phi^{-1})^\ast$ is an isomorphism of pre-Lie algebras.
\end{proof}

\begin{rmk}
In general, homomorphisms of  pre-Lie bialgebras and coherent homomorphisms of pre-Lie bialgebras are not related. For example,  when $\g=\h, f+f^\ast$ is usually not an endomorphism of the  pre-Lie algebra $\g\bowtie\g^\ast$.
\end{rmk}

Next we consider another notion of homomorphisms of Manin triples of pre-Lie algebras, originated  from the notion of isomorphisms of  Manin triples.
\begin{defi}\label{defi: strong homomorphism}
Let $((\g\bowtie\g^\ast,(\cdot,\cdot)_-) ,\g, \g^\ast)$ and $((\h\bowtie\h^\ast,(\cdot,\cdot)_-) ,\h, \h^\ast)$ be Manin triples of pre-Lie algebras. A {\bf  homomorphism } between them is a pre-Lie algebra homomorphism
\begin{eqnarray*}
f:\g\bowtie\g^\ast\rightarrow\h\bowtie\h^\ast,
\end{eqnarray*}
that restricts to pre-Lie algebra homomorphisms
\begin{eqnarray*}
f\mid_\g:\g\rightarrow\h, f\mid_{\g^\ast}:\g^\ast\rightarrow\h^\ast,
 \end{eqnarray*}
and satisfies
 homomorphism
\begin{eqnarray}\label{eq:compatible with the linear form}
(x+a, y+b)_-=(f(x+a), f(y+b))_-,~~\forall x+a, y+b\in \g\bowtie\g^\ast.
\end{eqnarray}
\end{defi}

\begin{pro}\label{pro:Manin triples}
Let$((\g\bowtie\g^\ast,(\cdot,\cdot)_-) ,\g, \g^\ast)$ and $((\h\bowtie\h^\ast,(\cdot,\cdot)_-) ,\h, \h^\ast)$  be two Manin triples of pre-Lie algebras. Let $(\phi, \psi)$ be a  homomorphism between them. Then $\psi\phi={\rm id}$, that is, $\phi$ is injective and $ \psi$ is  surjective.
\end{pro}
\begin{proof}
By  \eqref{eq:compatible with the linear form}, for any $x\in \g, a\in\g^\ast $ we have
\begin{eqnarray*}
(x, a)_-=(\phi(x), \psi^\ast(a))_-
=-\langle\phi(x),  \psi^\ast(a)\rangle
=-\langle \psi(\phi(x)), a\rangle
=( \psi(\phi(x)), a)_-     ,
\end{eqnarray*}
which implies that $ \psi\phi={\rm id}.$
\end{proof}

\subsection{Coherent homomorphism of classical $\frks$-matrices}
We now build the relation between classical $\frks$-matrices and bialgebras for pre-Lie algebras to the level of categories. B

As in the case of endo pre-Lie bialgebras, solutions of the $ \psi$-CYBE in endo pre-Lie algebra motivates us to the notion of  morphisms of  classical $\frks$-matrices that is compatible with coherent homomorphisms of pre-Lie bialgebras.
\begin{defi}\label{defi:coherent homomorphism of classical Yang-Baxter equation}
Let $\g, \h$ be pre-Lie algebras and $r_\g, r_\h$ be classical  $\frks$-matrices in $\g$ and $\h$  respectively.  A {\bf coherent homomorphism} from $r_\g$ to $r_\h$ consist of a  pre-Lie algebra homomorphism $\phi:\g\rightarrow\h$ and a linear map $ \psi:\h\rightarrow\g$ satisfying
\begin{eqnarray}
( \psi\otimes {\rm id}_\h)r_\h&=&({\rm id}_\g\otimes\phi)r_\g,\label{eq:coherent homomorphism  1}\\
({\rm id}_\h\otimes  \psi)r_\h&=&(\phi\otimes {\rm id}_\g)r_\g,\label{eq:coherent homomorphism  2}\\
\psi(y\cdot_\h \phi(x))&=&\psi(y) \cdot_\g x ,\label{eq:coherent homomorphism  21}\\
 \psi(\phi(x)\cdot_\h y)&=&x\cdot_\g  \psi(y)\label{eq:coherent homomorphism  3}, ~\forall x\in\g, y \in\h.
\end{eqnarray}
\end{defi}

If $\phi$ and $ \psi$ are also isomorphisms. Then $(\phi,  \psi)$ is called a {\bf coherent isomorphism} from $r_\g$ to $r_\h$. Let {\bf Cs} denote the category of classical $\frks$-matrices with the coherent morphisms.

Then by Definitions \ref{defi:classical Yang-Baxter equation} and \ref{defi:coherent homomorphism of classical Yang-Baxter equation}, we have
\begin{pro}
Let $(\g,\cdot_\g, \phi)$ be an endo pre-Lie algebra and $ \psi\in \End(\g)$ dual represent $(\g, \phi)$. Then $r\in\g\otimes\g$ is a solution of $\psi$-CYBE in the endo pre-Lie algebra $(\g,\cdot_\g, \phi)$  if and only if $r\in\g\otimes\g$ is a  classical $\frks$-matrices, and $(\phi,\psi)$ is a coherent endomorphism on $r$.
\end{pro}

Two classical $\frks$-matrices $r_1$ and $r_2$ in a pre-Lie algebra $\g$ are said to be {\bf equivalent} if there is a pre-Lie algebra isomorphism $\phi:\g\rightarrow\g$ such that $(\phi\otimes \phi)r_1=r_2$.
\begin{cor}
Let $\g$ be a pre-Lie algebra and $r_1$, $r_2$ be classical $\frks$-matrices in $\g$. Then $r_1$ is equivalent to $r_2$ if and only if there exists a pre-Lie algebra isomorphism $\phi:\g\rightarrow\g$ such that $(\phi, \phi^{-1})$ is  a coherent isomorphism from $r_1$ to $r_2$.
\end{cor}

\begin{proof}
By the fact that $r_1$ is equivalent to $r_2$, $(\phi, \phi^{-1})$ satisfying \eqref{eq:coherent homomorphism  1}-\eqref{eq:coherent homomorphism  3}. Conversely, by \eqref{eq:coherent homomorphism  1}, we have $(\phi\otimes \phi)r_1=r_2$, which implies that $r_1$ is equivalent to $r_2$ if and only if there exists a pre-Lie algebra isomorphism $\phi:\g\rightarrow\g$ such that $(\phi, \phi^{-1})$ is a coherent isomorphism from $r_1$ to $r_2$.
\end{proof}

Recall that for a pre-Lie algebra $\g$ and a classical $\frks$-matrix $r$ in $\g$ satisfying \eqref{eq:coboundary condition 1}, the $(\g, \g^\ast, \alpha, \beta)$ is a quasi-triangular pre-Lie bialgebra. On the level of categories, we have
\begin{thm}
Let $\g$, $\h$ be pre-Lie algebras and $r_\g,r_\h$ be classical  $\frks$-matrices in $\g$ and $\h$ respectively satisfying \eqref{eq:coboundary condition 1}. If $(\phi, \psi)$ is a coherent homomorphism of classical $\frks$-matrices from $r_\g$ to $r_\h$, then  $(\phi, \psi)$ is a coherent homomorphism of corresponding pre-Lie bialgebras from $(\g, \g^\ast, \alpha_{r_\g}, \beta_\g)$ to $(\h, \h^\ast, \alpha_{r_\h}, \beta_\h)$.
\end{thm}

\begin{proof}
 By the proof of Theorem \ref{thm:equivalent theorem}, for any $x\in\h$, $ r_\g=\Sigma(a_i\otimes b_i)$ and $r_\h=\Sigma(\xi_i\otimes\eta_i)$, we have
 \begin{eqnarray*}0&=&( \psi\otimes \psi)\alpha_{r_\h}(x)-\alpha_{r_\g} (\psi(x))\\
&=&( \psi\otimes \psi)(L_x\otimes {\rm id}_\h+{\rm id}_\h\otimes\ad_x)\Sigma(\xi_i\otimes\eta_i)-(L_{ \psi(x)}\otimes {\rm id}_\g+{\rm id}_\g\otimes\ad_{ \psi(x)})\Sigma( a_i\otimes b_i)\\
&=&( \psi\otimes \psi)\Sigma(x\cdot_\h \xi_i\otimes\eta_i)+\xi_i\otimes[x, \eta_i])-\Sigma( \psi(x)\cdot_\g a_i\otimes b_i+a_i\otimes[ \psi(x), b_i]_\g)\\
&=&\Sigma( \psi(x\cdot_\h\xi_i)\otimes \psi(\eta_i)+ \psi(\xi_i)\otimes \psi([x, \eta_i]_\h)- \psi(x)\cdot_\g a_i\otimes b_i-a_i\otimes[ \psi(x), b_i]_\g)\\
&=&( \psi L_x\otimes {\rm id}_\g)({\rm id}_\h\otimes \psi)r_\h+({\rm id}_\g\otimes \psi\ad_x)( \psi\otimes {\rm id}_\h)r_\h-\Sigma( \psi(x\cdot_\h\phi(a_i))\otimes b_i+a_i\otimes \psi([x, \phi(b_i)]_\g))\\
&=&( \psi L_x\otimes {\rm id}_\g)({\rm id}_\h\otimes \psi)r_\h+({\rm id}_\g\otimes \psi\ad_x)( \psi\otimes {\rm id}_\h)r_\h-( \psi L_x\otimes {\rm id}_\g)(\phi\otimes {\rm id}_\g)r_\g-({\rm id}_\g\otimes \psi\ad_x)({\rm id}_\g\otimes\phi)r_\g\\
&=&( \psi L_x\otimes {\rm id}_\g)(({\rm id}_\h\otimes \psi)r_\h-(\phi\otimes {\rm id}_\g)r_\g)+({\rm id}_\g\otimes \psi\ad_x)(( \psi\otimes {\rm id}_\h)r_\h-({\rm id}_\g\otimes\phi)r_\g),
\end{eqnarray*}
which implies that $( \psi\otimes \psi)\alpha_{r_\h}=\alpha_{r_\g} \psi $ if and only if
\begin{eqnarray*}
( \psi L_x\otimes {\rm id}_\g)(({\rm id}_\h\otimes \psi)r_\h-(\phi\otimes {\rm id}_\g)r_\g)+({\rm id}_\g\otimes \psi\ad_x)(( \psi\otimes {\rm id}_\h)r_\h-({\rm id}_\g\otimes\phi)r_\g)=0.
\end{eqnarray*}

By a direct calculation,  $({\rm id}_\g\otimes\phi)\alpha_{r_\g}=( \psi\otimes {\rm id}_\h)\alpha_{r_\h}\phi$  if and only if for any $x\in\g$,
\begin{eqnarray*}
( L_x\otimes {\rm id}_\g+{\rm id}_\h\otimes\ad_{\phi(x)})(({\rm id}_\g\otimes\phi)r_\g-( \psi\otimes {\rm id}_\h)r_\h))=0,
\end{eqnarray*}
and
$(\phi\otimes {\rm id}_\g)\alpha_{r_\g}=({\rm id}_\h\otimes  \psi)\alpha_{r_\h}\phi$ if and only if 
\begin{eqnarray*}
( L_x\otimes {\rm id}_\g+{\rm id}_\h\otimes\ad_{\phi(x)})((\phi\otimes{\rm id}_\g)r_\g-( {\rm id}_\h\otimes\psi)r_\h))=0.
\end{eqnarray*}

 Therefore,  $(\phi,  \psi)$ is a coherent homomorphism of pre-Lie bialgebras from $(\g, \g^\ast, \alpha_{r_\g}, \beta_\g)$ to  $(\h, \h^\ast, \alpha_{r_\h}, \beta_\h)$.
\end{proof}

\section{Cohomology and deformations of symmetrical classical $\frks$-matrices via relative Rota-Baxter operators }\label{sec:def-phase}
In this section, we first introduce the notion of weak homomorphism between symmetric classical $\frks$-matrices, which can induce an equivalence between categories of symmetric classical $\frks$-matrices and their corresponding relative Rota-Baxter operators. Furthermore, we use the graded Lie algebra for relative Rota-Baxter operators to give the graded Lie algebra for symmetric classical $\frks$-matrices. At last, we  give the cohomology of symmetric classical $\frks$-matrices and apply it to study their deformations.

\subsection{Homomorphism of symmetric $\frks$-matrices and relative Rota-Baxter oprators}
Let $\g$ be a vector space.  %Define$$\Sym^2(\g)=\{r\in\g\otimes\g\mid r(\xi,\eta)=r(\eta,\xi),\quad \forall~\xi,\eta\in\g^*\}. $$
For any $r\in\Sym^2(\g)$, the linear map $r^\sharp:\g^*\longrightarrow \g$ is given by
$\langle r^\sharp(a),b\rangle=r(a,b),$ for all $a,b\in\g^*$.  We say that $r\in\Sym^2(\g)$ is {\bf invertible}, if the linear map $r^\sharp$ is an isomorphism.

\begin{defi}
 A {\bf pseudo-Hessian  structure} on a pre-Lie algebra $(\g,\cdot_\g)$  is a symmetric nondegenerate $2$-cocycle $\frkB\in\Sym^2(\g^*)$.  More precisely,
$$\frkB(x\cdot_\g y,z)-\frkB(x,y\cdot_\g z)=\frkB(y\cdot_\g x,z)-\frkB(y,x\cdot_\g z),\quad \forall~x,y,z\in\g. $$
A pre-Lie algebra equipped with a pseudo-Hessian   structure  is called  a {\bf pseudo-Hessian pre-Lie algebra},
and denoted by $(\g,\cdot_\g,\frkB)$.
\end{defi}

\begin{pro}\label{pro:Hessian1}{\rm(\cite{Left-symmetric
bialgebras})} Let $(\frkg,\cdot_\frkg)$ be a pre-Lie algebra and
$r$ be an invertible classical $\frks$-matrix.  Define $\frkB\in \Sym^2(\g^*)$
by
\begin{equation}\label{eq:rB}
\frkB(x,y)=\langle x,(r^\sharp)^{-1}(y)\rangle,\quad \forall~x,y\in\g.
\end{equation}
Then $(\g,\cdot_\g,\frkB)$ is a pseudo-Hessian pre-Lie algebra.
\end{pro}

\begin{thm}{\rm(\cite{Left-symmetric bialgebras})}
Let $(\frkg,\cdot_\frkg)$ be a pre-Lie algebra and $r$ a symmetric classical $\frks$-matrix.  Then $(\g^*,\cdot_r)$ is a pre-Lie algebra, where the multiplication $\cdot_r:\otimes ^2\g^*\longrightarrow\g^*$ is  given by
\begin{equation}\label{eq:pre-bia}
a\cdot_{r}b=\ad^*_{r^\sharp(a)}b-R^*_{r^\sharp(b)}a,\quad a,b\in\g^*.
\end{equation}
Furthermore, $r^\sharp$ is a pre-Lie algebra homomorphism from $(\g^*,\cdot_r)$ to $(\frkg,\cdot_\frkg)$.
\end{thm}

The  sub-adjacent Lie algebra  of the pre-Lie algebra $(\g^*,\cdot_r)$ is $(\g^*,[\cdot,\cdot]_{r})$, where $[\cdot,\cdot]_{r}$ is given by
\begin{equation}\label{eq:commutator-H}
  [a,b]_{r}=L^*_{r^\sharp(a)}b-L^*_{r^\sharp(b)}a,\quad \forall~a,b\in\g^*.
\end{equation}

Furthermore, by a direct calculation, we have
\begin{pro}\label{pro:morphism}
For all $a,b\in\g^*$, we have
\begin{eqnarray}
\label{homo2}r^\sharp({[a,b]}_{r})-[r^\sharp(a), r^\sharp(b)]_\g&=&[[ r,r]](a,b,\cdot).
\end{eqnarray}
\end{pro}

Note that $L^*:\g\rightarrow \gl(\g^*)$  is a representation of the sub-adjacent Lie algebra  $\g^c$ on the dual space   $\g^*$.  Thus, by Proposition \ref{pro:morphism}, for $r\in\Sym^2(\g)$, we have
\begin{pro}\label{pro:LSBi-H}
$r$ is a symmetric classical $\frks$-matrix on a pre-Lie algebra $(\g,\cdot_\g)$ if and only if $r^\sharp:\g^*\rightarrow \g $ is a relative Rota-Baxter operator on the Lie algebra $\g^c$ with respect to the representation $(\g^*;L^*)$.
\end{pro}

The relations between coherent homomorphisms of symmetric classical $\frks$-matrices and homomorphisms between
the relative Rota-Baxter operators are given as following. 
\begin{thm} \label{pp:oprm}
Let $r_\g,\;r_\h$ be symmetric classical $\frks$-matrices in
pre-Lie algebras $\g$ and $\h$, respectively. Let $\phi:\g\rightarrow \h$ be a
pre-Lie algebra homomorphism and $\psi:\h\rightarrow \g$ be a linear map. If 
$(\phi, \psi)$ is a coherent homomorphism of classical $\frks$-matrices from $r_\g$
to $r_\h$, then $(\phi,\psi^*)$ is a homomorphism between
the corresponding relative Rota-Baxter operators $r_\g^\sharp$ and
$r_\h^\sharp$. 
\emptycomment{Conversely, if $(\phi,\psi^*)$ is a homomorphism between
the corresponding relative Rota-Baxter operators $r_\g^\sharp$ and
$r_\h^\sharp$ and $\psi$ satisfies \eqref{eq:the polarization 3}, then $(\phi, \psi)$ is a coherent homomorphism of classical $\frks$-matrices from $r_\g$
to $r_\h$.}

This correspondence defines a functor from the category ${\bf
SCs}$ of symmetric classical $\frks$-matrices to the category ${\bf SRB}$ of
relative Rota-Baxter operators on the sub-adjacent Lie algebras associated to the representations $L^*$ satisfying $\langle r^\sharp(a),b\rangle=\langle r^\sharp(b),a\rangle$ for $a,b\in \g^*$ or $\h^*$.
\end{thm}
\begin{proof}
Since $\phi:\g\rightarrow \h$ is a pre-Lie algebra homomorphism, then $\psi$ is a Lie algebra homomorphism from $\g^c$ to $\h^c$. For $x\in\g,y\in\h,a\in \g^*$, we have
\begin{eqnarray*}
\langle \psi^*L_x^*a, y\rangle&=&\langle
L_x^*a,\psi(y)\rangle=-\langle
a,x\cdot_\g\psi(y)\rangle,\\
\langle L^*_{\phi(x)}\psi^*(a),y\rangle&=&-\langle
\psi^*(a),\phi(x)\cdot_\h y\rangle=-\langle
a,\psi(\phi(x)\cdot_\h y)\rangle.
\end{eqnarray*}
Hence $\psi^*L_x^*a=L^*_{\phi(x)}\psi^*(a)$ if and
only if $x\cdot_\g\psi(y)=\psi(\phi(x)\cdot_\h y)$.

Let $a\in\g^*,b\in \h^*$. Then we have
\begin{eqnarray*}
\langle r_\h^\sharp\psi^*(a),b\rangle&=&\langle r_\h,
\psi^*(a)\otimes b\rangle=\langle r_\h,b\otimes
\psi^*(a)\rangle
=\langle ({\id}_\h\otimes \psi)(r_\h),b\otimes a\rangle,\\
\langle \phi r_\g^\sharp(a),b\rangle&=&\langle r_\g,
a\otimes \phi^*(b)\rangle=\langle r_\g,\phi^*(b)\otimes
a\rangle=\langle (\phi\otimes {\id}_\g)(r_\g),b\otimes
a\rangle.
\end{eqnarray*}
Hence $r_\h^\sharp\psi^*(b)=\phi r_\g^\sharp(a)$ if and only
if $(\psi\otimes {\id}_\h)(r_\h)=({\id}_\g\otimes \phi)(r_\g)$, and if
and only if  $({\id}_\h\otimes \psi)(r_\h)=(\phi\otimes
{\id}_\g)(r_\g)$. Therefore, $(\phi,\psi^*)$ is a homomorphism between
the corresponding relative Rota-Baxter operators $r_\g^\sharp$ and
$r_\h^\sharp$.

The rest can be obtained directly. 
\end{proof}

Motivated by Theorem \ref{pp:oprm}, we introduce the weak homomorphism between symmetric classical  $\frks$-matrices, which can induce the equivalence between these two categories.
\begin{defi}\label{defi:weak homomorphism of s-matrix}
Let $\g, \h$ be pre-Lie algebras and $r_\g, r_\h$ be symmetric classical  $\frks$-matrices in $\g$ and $\h$  respectively.  A {\bf weak homomorphism} from $r_\g$ to $r_\h$ consists of a  Lie algebra homomorphism $\phi:\g^c\rightarrow\h^c$ and a linear map $ \psi:\h\rightarrow\g$ satisfying
\begin{eqnarray}
( \psi\otimes {\rm id}_\h)r_\h&=&({\rm id}_\g\otimes\phi)r_\g,\label{eq:weak 1}\\
\psi(\phi(x)\cdot_\h y)&=&x\cdot_\g\psi(y),\quad x\in \g,y\in\h.\label{eq:weak 2} 
\end{eqnarray}
\end{defi}

It is obvious that a coherent homomorphism from symmetric classical  $\frks$-matrices $r_\g$ to $r_\h$ is a weak homomorphism from $r_\g$ to $r_\h$.

The following proposition shows that the notion of homomorphisms of
relative Rota-Baxter operators is compatible with the weak homomorphism between symmetric classical $\frks$-matrices. 
\begin{pro} \label{pro:weak-homomorphsim}
Let $r_\g,\;r_\h$ be symmetric classical $\frks$-matrices in
pre-Lie algebras $\g$ and $\h$, respectively. Let $\phi:\g\rightarrow \h$ and $\psi:\h\rightarrow \g$ be linear maps. Then  
$(\phi, \psi)$ is a weak homomorphism of symmetric classical $\frks$-matrices from $r_\g$
to $r_\h$ if and only if $(\phi,\psi^*)$ is a homomorphism between
the corresponding relative Rota-Baxter operators $r_\g^\sharp$ and $r_\h^\sharp$. 

This correspondence defines an equivalence from the category ${\bf
SCs}$ of symmetric classical $\frks$-matrices to the category ${\bf SRB}$ of
relative Rota-Baxter operators on the sub-adjacent Lie algebras associated to the representations $L^*$ satisfying $\langle r^\sharp(a),b\rangle=\langle r^\sharp(b),a\rangle$ for $a,b\in \g^*$ or $\h^*$.
\end{pro}

Let $\g$ be a pre-Lie algebra and $r\in  \Sym^2(\g)$ be a classical $\frks$-matrix.   Then $(\g^*,[\cdot,\cdot]_{r})$ is a Lie algebra, where the bracket $[\cdot,\cdot]_{r}$ is given by \eqref{eq:commutator-H}.  Furthermore, one has

\begin{pro}{\rm(\cite{Left-symmetric bialgebras})}\label{pro:para-kahler-phase-space}
 $({\g}\oplus {\g^*},[\cdot,\cdot]_p,\omega_p)$ is a symplectic Lie algebra, where the bracket $[\cdot,\cdot]_p$ is defined by
    \begin{eqnarray}\label{eq:parakahler bracket}
    [x+a,y+b]^r_p=[a,b]_{r}+L^*_xb-L^*_y a+\frkL^*_a y-\frkL^*_b x+[x,y]_\g,\quad x,y\in\g,a,b\in\g^*.
    \end{eqnarray}
   We call it the phase space of the Lie algebra $\g^c$ associated to the classical $\frks$-matrix $r$ and denote it by $(\g,{\g^*},\omega_p,r)$.
\end{pro}

\emptycomment{\begin{defi}
Let $(\g,{\g^*},\omega_p,r_1)$ and $(\g,{\g^*},\omega_p,r_2)$ be the phase spaces of the Lie algebra $\g^c$ associated to the symmetric classical $\frks$-matrices $r_1$ and $r_2$ on the pre-Lie algebra $\g$, respectively. A {\bf weak homomorphism} between them is a Lie algebra homomorphism
\begin{eqnarray*}
f:\g^c\bowtie{\g_{r_2}^\ast}^c\rightarrow\g^c\bowtie{\g_{r_1}^\ast}^c,
\end{eqnarray*}
 that restricts to Lie algebra homomorphisms
 \begin{eqnarray*}
f\mid_\g:\g^c\rightarrow\g^c, \quad f\mid_{\g^\ast}:{\g_{r_2}^\ast}^c\rightarrow{\g_{r_1}^\ast}^c.
 \end{eqnarray*}
\end{defi}}

\begin{pro}\label{eq:s-matrix-phase}
Let $\g$ be a pre-Lie algebra and $r_1,r_2\in  \Sym^2(\g)$ two classical $\frks$-matrices.  If $(\phi, \psi)$ is a weak homomorphism from $r_2$ to $r_1$, then $\phi+ \psi^\ast$ is a Lie algebra homomorphism from the Lie algebra $\g^c\bowtie{\g_{r_2}^\ast}^c$ to $\g^c\bowtie{\g_{r_1}^\ast}^c$.
\end{pro}
\begin{proof}
Since $\phi$ is a pre-Lie algebra homomorphism from $\g$ to $\g$, $\phi$ is a Lie algebra homomorphism from $\g^c$ to $\g^c$.  By \eqref{eq:weak 1} and \eqref{eq:weak 2}, we have
\begin{eqnarray*}
  \langle \psi^*[a,b]_{r_2},x\rangle &=&-\langle b, \psi(\phi(r_2^\sharp(a))\cdot_\g x)\rangle+\langle a, \varphi(\phi(r_2^\sharp(b))\cdot_\g x)\rangle\\
  &=& -\langle b, \psi(r_1^\sharp \psi^*(a)\cdot_\g x)\rangle+\langle a, \psi(r_1^\sharp\psi( \varphi^*(b))\cdot_\g x)\rangle\\
   &=&-\langle \psi^*b,r_1^\sharp \psi^*(a)\cdot_\g x\rangle+\langle \psi^*a,r_1^\sharp\psi( \psi^*(b))\cdot_\g x\rangle\\
   &=&\langle L_{r_1^\sharp \psi^*(a)} \psi^*b-L_{r_1^\sharp \psi^*(b)} \psi^*a,x\rangle\\
   &=&\langle[ \psi^*(a), \psi^*(b)]_{r_1},x\rangle,
\end{eqnarray*}
which implies that $ \psi^*$ is a Lie algebra homomorphism from $(\g^*,[\cdot,\cdot]_{r_2})$ to $(\g^*,[\cdot,\cdot]_{r_1})$.  Next, we show that $(\phi+ \psi^*)[x,a]_2=[\phi(x), \psi^*(a)]_1,$ which is equivalent to
\begin{eqnarray}
\label{eq:LA1}\phi(\frkL_a^* x)&=&\frkL^*_{ \psi^*(\alpha)}\phi(x),\\
\label{eq:LA2} \psi^*(L^*_xa)&=&L^*_{\phi(x)} \psi^*(a),\quad \forall~x\in\g,a\in \g^*.
\end{eqnarray}
One can check that \eqref{eq:LA1} follows from $\phi$ is a Lie algebra homomorphism, \eqref{eq:weak 1} and \eqref{eq:weak 2}. And  \eqref{eq:LA2} follows from \eqref{eq:weak 2}.
\end{proof}

\subsection{Cohomology and deformations of symmetric classical $\frks$-matrices}\label{sec:Coh-def}
In  this section, we use the graded Lie algebra for relative Rota-Baxter operators to give the graded Lie algebra whose Maurer-Cartan elements characterize classical $\frks$-matrices on pre-Lie algebras.

By Theorem \ref{pro:gla} and Proposition \ref{pro:LSBi-H}, we have
\begin{pro}\label{lem:graded Lie algbra}
Let $(\frkg,\cdot_\frkg)$ be a pre-Lie algebra and $r\in \Sym^2(\g)$.
\begin{itemize}
\item[{\rm (i)}]  $(\huaC^*(\g^*,\g):=\oplus_{k\geq0}\Gamma(\Hom(\wedge^{k}\g^*,\g)),\Courant{\cdot,\cdot})$ is a graded Lie algebra, where the bracket $\Courant{\cdot,\cdot}$ is given by \eqref{o-bracket} with $\rho=L^*$.
\item[{\rm (ii)}]  $r$ is a classical $\frks$-matrix on the pre-Lie algebra $\g$ if and only if $r^\sharp$ is a Maurer-Cartan element of the graded Lie algebra $(\huaC^*(\g^*,\g),\Courant{\cdot,\cdot})$.
\end{itemize}

\end{pro}

For $k\geq0$, define $\Psi:\wedge^{k}\g\otimes \g\longrightarrow\huaC^k(\g^*,\g) $ by
\begin{equation}\label{eq:relation-coboundary}  \langle\Psi( \psi)(a_1,\cdots,a_k),a_{k+1}\rangle=\langle \psi,a_1\wedge\cdots\wedge a_k\otimes a_{k+1}\rangle,\quad \forall~a_1,\cdots,a_{k+1}\in\g^*,
\end{equation}
and $\Upsilon:\huaC^k(\g^*,\g) \longrightarrow \wedge^{k}\g\otimes \g$ by
\begin{equation*}\label{eq:defiUpsilon}
\langle \Upsilon(P),a_1\wedge\cdots\wedge a_k\otimes a_{k+1}\rangle=\langle P(a_1,\cdots,a_k),a_{k+1}\rangle,\quad \forall~a_1,\cdots,a_{k+1}\in\g^*.
\end{equation*}
Obviously we have $\Psi\circ\Upsilon={\Id},~~\Upsilon\circ\Psi={\Id}. $

By Lemma \ref{lem:graded Lie algbra} , we have
\begin{thm}\label{thm:MC char}
Let $(\frkg,\cdot_\frkg)$ be a pre-Lie algebra.   Then, there is a graded Lie  bracket $\llbracket \cdot,\cdot\rrbracket_{\frks}:(\wedge^{k}\g\otimes \g)\times (\wedge^{l}\g\otimes \g)\longrightarrow \wedge^{k+l}\g\otimes \g$ on the graded vector space $C_{ \frks}^{*}(\g):=\oplus_{k\ge 1}C_{\rm \frks}^{k}(\g)$ with $C_{ \frks}^{k}(\g):=\wedge^{k-1}\g\otimes \g$ given by
$$
\llbracket \psi,\phi\rrbracket_{\frks}:=\Upsilon\llbracket \Psi( \psi),\Psi(\phi)\rrbracket,\quad  \psi\in\wedge^{k}\g\otimes \g,\phi\in\wedge^{l}\g\otimes \g.
$$
Furthermore,   $r\in \Sym^2(\g)$  is an $\frks$-matrix on the pre-Lie algebra $\g$ if and only if $r$ is a Maurer-Cartan element of the graded Lie algebra $(C_{\frks}^{*}(\g),\Courant{\cdot,\cdot}_{\frks})$.  More precisely, we have
  \begin{equation}
    \llbracket r,r\rrbracket_{\frks}(a_1,a_2,a_3)=2[[r,r]](a_1,a_2,a_3),\quad\forall~ a_1,a_2,a_3\in\g^*.
  \end{equation}
\end{thm}

In the following, we first give the cohomology of classical $\frks$-matrices on pre-Lie algebras.  Then we build the relationships between cohomology of a classical $\frks$-matrix and the corresponding relative Rota-Baxter operator.  At last, we use this cohomology to study the one-parameter infinitesimal deformations of classical $\frks$-matrices.

Let $r\in \Sym^2(\g)$ be a classical $\frks$-matrix on a pre-Lie algebra $(\g,\cdot_\g)$.  Set $C_{\frks}^{k}(\g)=\wedge^{k-1}\g\otimes \g$ and $C_{\frks}^{*}(\g)=\sum_{k\geq 1}C_{\frks}^{k}(\g)$.  Define $\delta_{\frks}:C_{\frks}^{k}(\g)\longrightarrow
C_{\frks}^{k+1}(\g)$ by
\begin{equation}
\delta_{\frks} \psi =(-1)^{k-1}\Courant{r, \psi}_{\frks},\quad  \psi\in C_{\frks}^{k}(\g).
\end{equation}
By the graded Jacobi identity, we have $\delta_{\frks}\circ \delta_{\frks}=0$.  Thus  $(C_{\frks}^{*}(\g),\delta_{\frks})$ is a cochain complex.  Denote by $H_{\frks}^k(\g)$ the $k$-th cohomology group.

Furthermore, we have
\begin{pro}For $ \psi\in C_{\frks}^{k}(\g)$, we have
 \begin{eqnarray}\label{eq:coboundary of A*}
\nonumber\delta_{\frks} \psi(a_1,\cdots,a_{k+1})
 &=&-\sum_{i=1}^{k}(-1)^{i+1} \psi(a_1, \cdots,\hat{a_i},\cdots,a_k,a_i\cdot_{r} a_{k+1})\nonumber\\
 &&+\sum_{1\leq i<j\leq k}(-1)^{i+j} \psi([a_i,a_j]_{r},a_1,\cdots,\hat{a_i},\cdots,\hat{a_j},\cdots,a_{k+1}),
\end{eqnarray}
where $a_1,\cdots,a_{k+1}\in\g^*$, $\cdot_{r}$ is given by \eqref{eq:pre-bia} and $[\cdot,\cdot]_{r}$ is given by \eqref{eq:commutator-H}.
\end{pro}

It is clear that the coboundary operator $\delta_{\frks}$ is just the coboundary operator given by \eqref{eq:pre-Lie cohomology} associated to the pre-Lie algebra $(\g^*,\cdot_{r})$ with coefficients in trivial representations.

By Proposition \ref{pro:LSBi-H}, we have
\begin{pro}
  Let $r$ be a classical $\frks$-matrix on a pre-Lie algebra $(\g,\cdot_\g)$.  Then $(\g^*,\ast_{r})$ is a pre-Lie algebra, where $\ast_{r}$ is given by
  \begin{equation}
    a\ast_{r}b=L^*_{r^\sharp(a)}b,\quad \forall~a,b\in\g^*.
  \end{equation}
\end{pro}

 It is obvious that the pre-Lie algebras $(\g^*,\ast_{r})$ and $(\g^*,\cdot_{r})$ have the same sub-adjacent Lie algebra $(\g^*,[\cdot,\cdot]_{r})$.

By Lemma \ref{lem:rep}, we have
\begin{pro}\label{pro:KV-structure}
Let $r$ be a classical $\frks$-matrix on a pre-Lie algebra $(\g,\cdot_\g)$.
 Then
 \begin{equation}\label{eq:rep-coadjoint}
 \varrho:\g^*\longrightarrow \gl(\g),\quad \varrho(a)(x)=[r^\sharp(a),x]_{\g}+r^\sharp(L^*_xa ),\quad  x\in\g,a\in\g^*
\end{equation}
 is a representation of the sub-adjacent Lie algebra $(\g^*,[\cdot,\cdot]_{r})$ on the vector space $\g$.
\end{pro}

Note that the representation $ \varrho$ given by \eqref{eq:rep-coadjoint} is exactly the dual representation of the left  multiplication  operation of the  pre-Lie algebra $(\g^*,\ast_{r})$.

Let $r$ be a classical $\frks$-matrix on a pre-Lie algebra $(\g,\cdot_\g)$.  By \eqref{eq:odiff}, for $P\in\huaC^k(\g^*,\g)$ and $a_1,\cdots,a_{k+1}\in\g^*$, the coboundary operator $\delta_{\rm RB}:\huaC^k(\g^*,\g)\longrightarrow \huaC^{k+1}(\g^*,\g)$ of the relative Rota-Baxter operator $r^\sharp$ is given by
\begin{eqnarray}\label{eq:coboundary of A*A}
&&\delta_{\rm RB}P(a_1,\cdots,a_{k+1})\nonumber\\
&=&\sum_{i=1}^{k+1}(-1)^{i+1}[r^\sharp(a_i),P(a_1,a_2,\cdots,\hat{a_i},\cdots,a_{k+1})]_{\g}\nonumber\\
&&+\sum_{i=1}^{k+1}(-1)^{i+1}r^\sharp\big(L^*_{P(a_1,a_2,\cdots,\hat{a_i},\cdots,a_{k+1})}\alpha_i\big)\nonumber\\
&&+\sum_{1\leq i<j\leq {k+1}}(-1)^{i+j}P([a_i,a_j]_{r},a_1,\cdots,\hat{a_i},\cdots,\hat{a_j},\cdots,a_{k+1}).
\end{eqnarray}
Denote by $H^k(\g^*,\g)$ the $k$-th cohomology group, called the {\bf{ $k$-th cohomology group of the relative Rota-Baxter operator $r^\sharp$}}.

\begin{pro}\label{pro:isomorphism}
  With the above notations, the map $\Psi$ defined by \eqref{eq:relation-coboundary} is a cochain isomorphism between cochain complexes $(C_{\frks}^{*}(\g),\delta_{\frks})$ and $(\huaC^*(\g^*,\g),\delta_{\rm RB})$, i. e.  we have the following commutative diagram:
  \[
\small{ \xymatrix{
\cdots
\longrightarrow C_{\frks}^{k+1}(\g) \ar[d]^{\Psi} \ar[r]^{\quad\delta_{\frks}} & C_{\frks}^{k+2}(\g) \ar[d]^{\Psi} \ar[r]  & \cdots  \\
\cdots\longrightarrow \huaC^k(\g^*,\g) \ar[r]^{\quad \delta_{\rm RB}} &\huaC^{k+1}(\g^*,\g)\ar[r]& \cdots. }
}
\]
  Consequently, $\Psi$ induces an isomorphism map $\Psi_\ast$ between the corresponding cohomology groups.
\end{pro}
\begin{proof}
It follows by a direct calculation.
\end{proof}

Now we introduce a new cochain complex, whose cohomology groups  control   deformations of classical  $\frks$-matrices.  Let $r$ be a classical $\frks$-matrix on a pre-Lie algebra.  For all $a_1,a_2,a_3\in \g^*$, define
\begin{eqnarray*}
	\tilde{\huaC}_{\frks}^1(\g)&=&\{x \in 	{\huaC}_{\frks}^1(\g)\mid (R_x\otimes 1+1\otimes R_x)r=0\},\\
	\tilde{\huaC}_{\frks}^2(\g)&=&\{ \psi\in {\huaC}_{\frks}^2(\g)\mid \psi(a_1,a_2)= \psi(a_2,a_1)\},\\
	\tilde{\huaC}_{\frks}^3(\g)&=&\{ \psi\in {\huaC}_{\frks}^3(\g)\mid \psi(a_1,a_2,a_3)+c. p. =0\},\\
	\tilde{\huaC}_{\frks}^k(\g)&=&{\huaC}_{\frks}^k(\g),\quad k\geq 4.
\end{eqnarray*}
It is straightforward to verify that the cochain complex $(\tilde{C}_{\frks}^{*}(\g),\delta_{\frks})$ is a subcomplex of the cochain complex $(C_{\frks}^{*}(\g),\delta_{\frks})$.  Denote by $\tilde{H}_{\frks}^k(\g)$ the $k$-th cohomology group, called the {\bf $k$-th cohomology group of the $\frks$-matrix $r$}

\begin{defi}
Let $\g$ be a pre-Lie algebra and $r\in  \Sym^2(\g)$ a classical  $\frks$-matrix.  If $r+t\kappa$ is also a classical  $\frks$-matrix for any $t\in\Real$, then we say that $\kappa$ generates a {\bf one-parameter infinitesimal deformation} of $r$.
\end{defi}

\begin{pro}
If $\kappa$ generates a one-parameter infinitesimal deformation of $r$, then $\kappa$ is a $2$-cocycle.
\end{pro}
\begin{proof}
Since $r+t\kappa$ is a classical  $\frks$-matrix for any $t\in\Real$, we have
$$\Courant{r+t\kappa,r+t\kappa}_{\frks}=\Courant{r,r}_{\frks}+2t\Courant{r,\kappa}_{\frks}+t^2 \Courant{\kappa,\kappa}_{\frks}=0,$$
which implies that $\Courant{r,\kappa}_{\frks}=0$ and thus $\delta_\frks\kappa=0$.
\end{proof}

\begin{defi}
Two one-parameter infinitesimal deformations $r^1_t=r+t\kappa_1$ and $r^2_t=r+t\kappa_2$ of a classical  $\frks$-matrix $r$ are called {\bf equivalent} if there exists $x\in \tilde{\huaC}_{\frks}^1(\g)$, such that $({\Id_\g}+t\ad_x,{\Id_\g}-tL_x):r^2_t\rightarrow r^1_t$
 is a weak homomorphism.
\end{defi}

\begin{pro}
If two one-parameter infinitesimal deformations $r^1_t=r+t\kappa_1$ and $r^2_t=r+t\kappa_2$ of a classical $\frks$-matrix $r$ are equivalent, then $\kappa_1$ and $\kappa_2$ are in the same cohomology class of $\tilde{H}_{\frks}^2(\g)$.
\end{pro}
\begin{proof}
Since $({\Id_\g}+\ad_x,{\Id_\g}-tL_x):r^2_t\rightarrow r^1_t$ is a weak homomorphism,  there exists $x\in \tilde{\huaC}_{\frks}^1(\g)$ such that
$${\Id_\g} \otimes ({\Id_\g} +t\ad_x)(r^2_t)=({\Id_\g}-tL_x)\otimes {\Id_\g} (r^1_t),$$
which implies that $\kappa_2-\kappa_1=-({\Id_\g}\otimes \ad_x+L_x\otimes {\Id_\g})(r)=\delta_\frks(x). $
\end{proof}

\begin{defi}
 A one-parameter infinitesimal deformation $r_t=r+t\kappa$ is called {\bf trivial} if there exists $x\in \tilde{\huaC}_{\frks}^1(\g)$, such that $({\Id_\g}+\ad_x,{\Id_\g}-tL_x):r^2_t\rightarrow r$
 is a weak homomorphism.
\end{defi}

Let  $r_t=r+t\kappa$ be a trivial deformation of a classical  $\frks$-matrix $r$.  Then there exists $x\in \tilde{\huaC}_{\frks}^1(\g)$ such that $({\Id_\g}+\ad_x,{\Id_\g}-tL_x):r^2_t\rightarrow r$ is a weak homomorphism.  First, by the fact that ${\Id_\g}+\ad_x:\g^c\rightarrow\g^c$ is a Lie algebra homomorphism, we have
\begin{equation}
\label{eq:Nij1}[[x,y]_\g,[x,z]_\g]_\g=0,\quad y,z\in\g.
\end{equation}

By \eqref{eq:weak 1}, we have
\begin{equation}
\label{eq:Nij2}({\Id}\otimes\ad_x)({\Id}\otimes\ad_x+L_x\otimes {\Id})r=0.
\end{equation}

By \eqref{eq:weak 2}, we have
\begin{equation}
\label{eq:Nij3}x\cdot_\g([x,y]_\g\cdot_\g z )=0,\quad y,z\in\g.
\end{equation}

\begin{defi}
Let $\g$ be a pre-Lie algebra and $r\in  \Sym^2(\g)$ a classical  $\frks$-matrix.  An element $x\in \tilde{\huaC}_{\frks}^1(\g)$ is called a {\bf Nijenhuis element} if $x$ satisfies \eqref{eq:Nij1}-\eqref{eq:Nij3}.
\end{defi}

Thus, a trivial deformation of a classical  $\frks$-matrix gives rise to a Nijenhuis element.  Conversely, we have
\begin{pro}\label{eq:trivial deformation}
Let $\g$ be a pre-Lie algebra and $r\in  \Sym^2(\g)$ a classical  $\frks$-matrix.  Then for any Nijenhuis element $x\in \tilde{\huaC}_{\frks}^1(\g)$, $r_t=r+t\Courant{r,x}_\frks$ is a trivial one-parameter infinitesimal deformation of $r$.
\end{pro}
\begin{proof}
Since $r\in  \Sym^2(\g)$ is a classical  $\frks$-matrix and $x\in \tilde{\huaC}_{\frks}^1(\g)$, by graded Jacobi identity, we have
$$\Courant{\Courant{r,x}_\frks,\Courant{r,x}_\frks}_\frks=0,$$
which implies that $\Courant{r,x}_\frks\in \Sym^2(\g)$ is also a classical  $\frks$-matrix.  Also, we have $\Courant{r,\Courant{r,x}_\frks}_\frks=0$.  Thus $r_t=r+t\Courant{r,x}_\frks$ is a one-parameter infinitesimal deformation of $r$.

By \eqref{eq:Nij1},  ${\Id_\g}+\ad_x:\g^c\rightarrow\g^c$ is a Lie algebra homomorphism.  By \eqref{eq:Nij2}, \eqref{eq:weak 1} follows.  By \eqref{eq:Nij3}, \eqref{eq:weak 2} follows.  Therefore, the one-parameter infinitesimal deformation $r_t=r+t\Courant{r,x}_\frks$ is trivial.
\end{proof}

 \end{document}